\documentclass[review]{elsarticle2}

\usepackage[margin=1.1in]{geometry}
\usepackage{mathrsfs}
\usepackage{amsfonts}
\usepackage{amsthm}
\usepackage{paralist}
\usepackage{graphics}
\usepackage{epstopdf}
\usepackage{float}
\usepackage{graphicx,subfigure}
\usepackage{amsmath}
\usepackage{bm}
\usepackage[colorlinks=true]{hyperref}
\usepackage{color}
\usepackage{epstopdf}
\hypersetup{urlcolor=blue, citecolor=red}
\graphicspath{{./figs/}}
\usepackage{lineno,hyperref}
\nolinenumbers

\usepackage{filecontents}
\usepackage{pgfplots}
\pgfplotsset{compat=newest}
\usetikzlibrary{decorations.markings}
\usepackage{pgfplotstable}

\usepackage[normalem]{ulem}
\graphicspath{{./figures/}}
\usepackage{soul}
\soulregister\cite7
\soulregister\eqref7
\soulregister\autoref7
\usepackage{natbib}
\makeatletter
\def\ps@first{%
   \let\@oddhead\@empty
   \let\@evenhead\@empty
   \def\@oddfoot{}
   \let\@evenfoot\@oddfoot
}

\newtheorem{theorem}{Theorem}[section]

\newtheorem{cor}{Corollary}[section]

\numberwithin{equation}{section}

\newtheorem{rem}{Remark}[section]

\def\eu{e_{\mathbf{u}}}
\def\eV{e_{\mathbf{V}}}
\def\eQ{e_Q}
\def\ep{e_p}
\def\eq{e_q}
\def\Av{\mathbf{V}}

\def\ub{\bm{u}}
\def\vb{\bm{v}}

\allowdisplaybreaks

\numberwithin{equation}{section}

\date{}

\begin{document}
	\begin{frontmatter}
		
		\title{Long-time stable SAV-BDF2 numerical schemes for the forced Navier-Stokes equations}
		
		%
		

		\author[a]{Daozhi Han
		}
		\ead{daozhiha@buffalo.edu}

		\author [b]{Xiaoming Wang}
		\ead{wxm@eitech.edu.cn}

		\address[a]{Department of Mathematics, State University of New York at Buffalo, Buffalo, NY 14260}
		\address[b]{School of Mathematical Science, Eastern Institute of Technology, Ningbo, Zhejiang 315200, China}


		\begin{abstract}
		We propose a novel second-order accurate, long-time unconditionally stable time-marching scheme  for the forced Navier-Stokes equations. A new Forced Scalar Auxiliary Variable approach (FSAV) is introduced to preserve the underlying dissipative structure of the forced system that yields a uniform-in-time estimate of the numerical solution. In addition, the numerical scheme is autonomous if the underlying model is, laying the foundation for studying long-time dynamics of the numerical solution via dynamical system approach. As an example we apply the new algorithm to the two-dimensional incompressible Navier-Stokes equations. In the case with no-penetration and free-slip boundary condition on a simply connected domain, we are also able to derive a uniform-in-time estimate of the vorticity in $H^1$ norm in addition to the $L^2$ norm guaranteed by the general framework. Numerical results demonstrate superior performance of the new algorithm in terms of accuracy, efficiency, stability and robustness.
		
		\end{abstract}
		
		\begin{keyword}
		 forced autonomous dynamical systems \sep Navier-Stokes equations \sep SAV \sep BDF2 \sep  long-time stability 
			\MSC[2010]  35K61
			\sep 76T99\sep  76S05\sep  76D07.
		\end{keyword}
		
	\end{frontmatter}
	
	\section{Introduction} \label{intro-sec}
Consider  the incompressible Navier-Stokes equations (NSE):
	\begin{align}
		\frac{\partial \ub}{\partial t}+\ub\cdot \nabla \ub-\frac{1}{Re} \Delta \ub=-\nabla p+\mathbf{F}, \quad \nabla \cdot \ub=0, \label{NSe}
	\end{align}
	equipped with either the no-slip, no-penetration boundary condition, or the no-penetration plus free slip boundary condition, or the periodic boundary condition. Throughout it is assumed that the $L^2$ norm of the body force $\mathbf{F}$ is uniformly bounded in time. Under these assumptions it is known that  in 2D $\ub$ is uniformly bounded-in-time in the $L^2$ norm (energy) and in the $H^1$ norm (enstrophy) for $H^1$ initial data, while in 3D the $L^2$ norm is bounded uniformly in time. The aim of this article is to develop efficient high-order accurate and long-time stable schemes that preserve the uniform-in-time estimates of $\ub$. 
	
The NSE is known for displaying chaotic and turbulent behavior \cite{Frisch1995, Temam1997, FMRT2001, MaWa2006a, MoYa2007} in the forced case at large Renolds number.  The long time dynamics could be highly non-trivial exhibiting intricate intermittent behavior even at relatively small Reynolds numbers with a simple Kolmogorov forcing \cite{Armbruster1996}. In the 2D case there exists a global attractor and invariant measures to the NSE, and it is understood that the long-time dynamics characterized by the global attractor and invariant measures  are critical to the understanding of chaos and turbulence. To numerically capture the long-time dynamics in terms of convergence of attractors and invariant measures, it is highly desirable that the numerical schemes are efficient, high-order accurate and  appropriately preserve the dissipativity of the NSE (hence long-time stable). Indeed, it is shown in \cite{Wang2010a}
that dissipativity-preserving is key for a first-order in time scheme to capture the long-time statistics of a dissipative dynamical system in the sense of convergence of the underlying invariant measures.

Early works  focus on fully implicit discretizations of the NSE which naturally preserve the uniform-in-time estimate  of the continuous solution in the energy norm in both 2D and 3D \cite{HeRa1990, ToWi2006, Tone2007}. Global in time error estimate is generally not expected due to the chaotic and turbulent behavior of the NSE. Under the assumption that the continuous solution of the NSE is exponentially stable, global-in-time error estimates of the Crank-Nicolson scheme were established in \cite{HeRa1990} where the discrete solution  was shown to be uniformly bounded in time in the $H^1$ norm, and in the $H^2$ norm with a time-step  constraint. Under time-step constraints the authors in \cite{ToWi2006, Tone2007} proved long-time stability in the $H^1$ norm for the Backward Euler scheme and the Crank-Nicolson scheme, respectively. Another popular family of time-marching schemes for the NSE is the implicit-explicit discretization scheme (IMEX) where the nonlinear advection term is treated semi-explicitly and the viscous term is discretized implicitly, so as to achieve unconditional long-time stability while maintaining reasonable efficiency for long-time simulations.  A general class of IMEX schemes including the fractional-step methods were considered in \cite{SiAr1994} and the large-time stability in the energy norm was shown.  For the 2D NSE in the vorticity-velocity formulation the authors in \cite{HOR2017} obtained the unconditional uniform-in-time estimate in the $H^1$ norm for the IMEX BDF2 scheme.  Recently such an estimate has also been proven for the IMEX BDF2 scheme applied directly to the 2D NSE in the primitive form, albeit with a mild time-step constraint independent of the spatial resolution \cite{ReTo2023}. See also \cite{HiSu2000} for a study on the same scheme, but in the fully discretized setting.
We would like to point out that the semi-implicit treatment of the advection term utilized in the IMEX schemes mentioned above leads to the need to solve a non-symmetric, non-constant coefficient linear problem at each time step.

From efficiency standpoint it is advantageous to discretize the nonlinear advection term fully explicitly so that the same coefficient matrix is shared throughout the time-marching for long-time simulations. For the 2D NSE in the vorticity-streamfunction formulation equipped with periodic boundary condition, long-time stability in the $L^2$ and $H^1$ norms and convergence of attractors and invariant measures at vanishing time step were proven for the first-order IMEX scheme in \cite{GTWW2012}. Subsequently the analysis was carried out for the IMEX-BDF2 scheme in \cite{Wang2012} where the author further established the uniform in time $H^2$ estimate as well as convergence of the marginal distributions of the invariant measures of the scheme to those of the NSE at vanishing time-step. Note that these results are restricted to the periodic boundary case and subject to a time-step constraint of the form $k \leq C Re^{-\alpha}, \alpha\ge 1$ independent of spatial resolution. It can be restrictive for large Reynolds number for capturing long time dynamics.

In recent years a Lagrange multiplier type approach, termed as the Scalar Auxiliary Variable method (SAV), has been developed for solving PDEs with a gradient flow structure \cite{SXY2018, SXY2019}, see also \cite{GuTi2013, YZW2017, YaJu2017, YaHa2017} for early variants of this method. High order SAV methods are constructed in \cite{ALL2019, GZW2020, HuSh2021, HuSh2022, WHS2022}. More recently, the SAV approach has been extended to treat the NSE and related fluid models, see for instance \cite{LYD2019, LiSh2020, LSL2021, JiYa2021, CHJ2023} among many others. 
In particular, the advective term of the NSE is treated explicitly in a SAV method and a prefactor (a scalar variable) is discretized implicitly. As a result, the SAV methods lead to the same Stokes solver across the time iterations. In the absence of an external forcing, it was shown that the SAV methods for the NSE are unconditionally stable in the sense that a (modified) energy is non-increasing, see \cite{LSL2021} and the references cited above. In the case of linear advection diffusion equation with constant coefficient, the stability of SAV method could be attributed to the reduction of the effective Pecl\'et number.  For fixed final time and with an external forcing, the SAV methods designed in \cite{LiSh2020, Yang2021} are also unconditionally stable. To our knowledge there are no existing SAV methods for the NSE  that can preserve the uniform in time estimate of the velocity in the energy norm, and long-time estimate in the $H^1$ norm in the 2D case in the presence of external forcing.

In this article we propose a novel SAV method based on the second-order backward differentiation formula (SAV-BDF2) for the NSE. Unlike previous SAV methods, our equation for the auxiliary scalar variable is damped, forced, and autonomous. For external forcing bounded uniformly in time, long-time stability bound in the $L^2$ norm is established for the new SAV-BDF2 methods in both 2D and 3D. Furthermore,   in the case of the 2D NSE in the vorticity-streamfunction formulation, it is shown that  the discrete enstrophy  and the $H^1$ norm of the vorticity are all bounded uniformly in time. When the external forcing is independent of time (autonomous), these uniform-in-time estimate, as well as the autonomous treatment of the auxiliary variable, may imply the existence of discrete global attractors to the dynamical system on the product space. Convergence of global attractors and invariant measures will be pursued in a separate work. It should be emphasized that the SAV-BDF2 methods render the same Brinkman equations or the Helmholtz equation in the 2D case at each time step. Therefore it is highly efficient for long-time computation. As far as we know the novel uniform-in-time SAV-BDF2 methods are so far the only methods that treat the nonlinear advection term of the NSE explicitly while preserving the uniform in time estimates without any time-step restriction in the presence of external forcing.

The rest of the article is organized as follows. In Sec. \ref{sec(2)} the general scheme and the scheme for 2D NSE in streamline formulation are introduced, and the long-time stability of the schemes are established. Error analysis for the general scheme in 2D is performed in Sec. \ref{sec(3)}. Numerical results related to accuracy, stability and robustness of the proposed methods are reported in Sec. \ref{s-nu}.  Some concluding remarks are provided in Sec. \ref{conc}.

	\section{The numerical scheme}\label{sec(2)}
\subsection{The general scheme}	One first introduces a scalar auxiliary variable (SAV) $q(t)$ such that
	\begin{align}
   \frac{dq}{dt}+\gamma q -\int_{\Omega} (\ub \cdot \nabla) \ub \cdot  \ub\, dx=\gamma, \quad q(0)=1, \label{SAVe}
	\end{align}
where $\gamma>0$ is a user-specified free parameter. It is  noted that $q(t)\equiv1$. One  then writes the NSe \eqref{NSe} as an expanded system
\begin{align}
    &\frac{\partial \ub}{\partial t}+q(t) \ub\cdot \nabla \ub-\frac{1}{Re} \Delta \ub=-\nabla p+\mathbf{F}, \quad \nabla \cdot \ub=0, \label{E-NSe1} \\
	&\frac{dq}{dt}+\gamma q -\int_{\Omega} (\ub \cdot \nabla) \ub \cdot  \ub\, dx=\gamma. \label{E-NSe2}
\end{align}

	Let $k>0$ be the time step size, $t^n=n k$ for an integer $n$, and denote the numerical approximation of $\ub$ at $t^n$ by $\ub^n$. The second order backward differentiation (BDF2) approximation of the derivative is then defined as
	\begin{align}
	\delta_t \ub^{n+1}:=\frac{3\ub^{n+1}-4\ub^n+\ub^{n-1}}{2k}. \label{BDF2-de}	
	\end{align}	
 Furthermore, the extrapolated approximation is denoted by $\overline{\ub}^{n+1}:= 2\ub^n-\ub^{n-1}$. 
 Finally,  the generalized $G$ matrix is defined as
\begin{align}
	G:=\frac{1}{4}\begin{bmatrix}
1 & -2\\
-2 & 5
	\end{bmatrix},
\end{align}
	and the associated (generalized) $G$-norm is defined as $||\mathbf{V}||_G:=\mathbf{V}\cdot G\mathbf{V}$ where $\mathbf{V}$ has two entries of either vectors or scalars. Recall the equivalency between the $G$-norms and standard norms
	\begin{align}
		&c_l ||\mathbf{V}||_G^2 \leq || \mathbf{V}|| \leq c_u || \mathbf{V}||_G^2, \label{equi1}\\
		&c_l |Q|_G^2 \leq |Q|^2 \leq c_u |Q|_G^2, \label{equi2}
	\end{align}
where $c_l, c_u$ are positive constants. 

We also assume that the boundary conditions ensure the validity of the Poincar\'e 's inequality, i.e., there exists $c_0$ such that $||\nabla\ub||^2\geq c_0||\ub||^2$.
	
Applying the modified BDF2 algorithm to the system \eqref{E-NSe1}--\eqref{E-NSe2}, 
one derives the BDF2-SAV scheme as follows
\begin{align}
	&\delta_t \ub^{n+1} +\frac{1}{Re} \Delta   \ub^{n+1}+q^{n+1} \big(\overline{\ub}^{n+1}\cdot \nabla \big) \overline{\ub}^{n+1} +\nabla p^{n+1}=\mathbf{F}^{n+1}, \quad \nabla \cdot \ub^{n+1}=0, \label{DDSe-NS}\\
	&\delta_t q^{n+1}+\gamma q^{n+1}-\int_\Omega  \big(\overline{\ub}^{n+1}\cdot \nabla \big) \overline{\ub}^{n+1} \cdot  \ub^{n+1}\, dx=\gamma. \label{DSAV-NS}
\end{align}	

Before establishing the long-time stability of the algorithm, one comments on how to solve the linear system. In light of the linearity one performs a linear superposition by setting
\begin{align*}
	\ub^{n+1}=\ub_1 + q^{n+1} \ub_2, \quad p^{n+1}=p_1+ q^{n+1} p_2,
\end{align*}
	where $(\ub_i, p_i)$ satisfies the Brinkman-type equations
	\begin{align*}
		\frac{3 \ub_i}{2k}+ \frac{1}{Re}\Delta \ub_i+ \nabla p_i=\mathbf{f}_i, \quad \nabla \cdot \ub_i=0, 
	\end{align*}
with
\begin{align}
	& \mathbf{f}_1:=\mathbf{F}^{n+1}+	\frac{4 \ub^n-\ub^{n-1}}{2k}; \quad \mathbf{f}_2:=-\big(2\ub^{n}-\ub^{n-1}\big)\cdot \nabla \big(2\ub^{n}-\ub^{n-1}\big).
\end{align}
Once $\ub_1$ and $\ub_2$ are available, one updates $q^{n+1}$ from the linear scalar Eq. \eqref{DSAV-NS}. The scheme is highly efficient because it is linear and the positive coefficient matrix does not change in time.

\begin{theorem}\label{theo1}
Suppose $\ub_0 \in \mathbf{H}^1_0(\Omega), \mathbf{F} \in \mathbf{L}^\infty\big(0,\infty; \mathbf{L}^2(\Omega)\big)$ and the scheme \eqref{DDSe-NS}--\eqref{DSAV-NS} is initiated by the first order counterpart. Then for any $k >0$  the scheme \eqref{DDSe-NS}	--\eqref{DSAV-NS} is long-time stable in the sense that $||\ub^n||+|q^n|\leq C$ for any $n\geq 2$ where $C$ is a constant dependent on $||\ub_0||_{H^1}$ and $||\mathbf{F}||_{L^\infty(0,\infty;L^2(\Omega))}$.	
\end{theorem}
\begin{proof}
	Testing Eq. \eqref{DDSe-NS} with $\ub^{n+1} k$, multiplying Eq. \eqref{DSAV-NS} by $q^{n+1} k$, and taking summation of the resultants, one obtains
	\begin{align}\label{es-1}
		&(|| \Av^{n+1}||_G^2+|Q^{n+1}|_G^2)-(|| \Av^{n}||_G^2+|Q^{n}|_G^2)+\frac{1}{4}||\ub^{n+1}-2\ub^n+\ub^{n-1} ||^2+\frac{1}{4}|q^{n+1}-2q^n+q^{n-1} |^2\nonumber \\
		&+\frac{k}{Re} || \nabla   \ub^{n+1}||^2+k \gamma |q^{n+1}|^2 \leq \frac{c_0}{2Re} k || \ub^{n+1}||^2+\frac{k Re}{2c_0}||\mathbf{F}||^2+\frac{k \gamma}{2}  |q^{n+1}|^2+\frac{k \gamma}{2 }.
	\end{align}
 which simplifies to
\begin{align}\label{es-3}
	&(|| \Av^{n+1}||_G^2+|Q^{n+1}|_G^2)+\frac{c_0}{4Re}k || \ub^{n+1}||^2+\frac{k \gamma}{2}  |q^{n+1}|^2 \leq (|| \Av^{n}||_G^2+|Q^{n}|_G^2)+\frac{Re k}{2c_0}||\mathbf{F}||^2+\frac{k \gamma}{2}.
\end{align}

 Recall the equivalence of norms in \eqref{equi1} and \eqref{equi2}. Let
 \begin{align}\label{conb-1}
     0<\beta\leq \min \{\frac{c_0}{8Re}, \frac{\gamma}{4}\}.
 \end{align}
 Adding $\beta k(||\ub^n||^2+|q^n|^2)$ to both sides of \eqref{es-3} yield
\begin{align}\label{es-4}
	&(|| \Av^{n+1}||_G^2+|Q^{n+1}|_G^2)(1+\beta c_l k)+\frac{c_0}{8Re}k || \ub^{n+1}||^2+\frac{k \gamma}{4}  |q^{n+1}|^2 \leq (|| \Av^{n}||_G^2+|Q^{n}|_G^2)\nonumber \\
	&+\beta k(||\ub^n||^2+|q^n|^2)
	+\frac{Re k}{2c_0}||\mathbf{F}||^2+\frac{k \gamma}{2}
\end{align}
Define
\begin{align}
	E^n:=|| \Av^n||_G^2+|Q^{n}|_G^2+\beta k(||\ub^n||^2+|q^n|^2). \label{dis-en}
\end{align}
Further imposing 
\begin{align}\label{conb-2}
    (1+\beta c_l k) \beta \leq \min \{\frac{c_0}{8Re}, \frac{\gamma}{4}\},
\end{align}
one obtains
\begin{align}\label{es-5}
	E^{n+1} (1+\beta c_l k) \leq E^n+ \frac{Re k}{2c_0}||\mathbf{F}||^2+\frac{k \gamma}{2}.
\end{align}

In light of the uniform-in-time boundedness of $F$, it follows that
\begin{align}\label{es-6}
	E^{n+1}  \leq \frac{1}{(1+\beta c_l k)^n}E^1+\frac{1}{2\beta c_l}  \big(\frac{Re}{c_0}||\mathbf{F}||_{L^\infty(0,\infty;L^2(\Omega))}^2+\gamma\big), \quad n=2, 3 \ldots.
\end{align}
This completes the proof.
\end{proof}

Assuming $k\leq 1$, one can choose the largest $\beta$ such that
\begin{align}\label{b-con}
  0<\beta\leq \min \{\frac{c_0}{8Re}, \frac{\gamma}{4}\}, \quad   (1+\beta c_l ) \beta \leq \min \{\frac{c_0}{8Re}, \frac{\gamma}{4}\}.
\end{align}
The estimate \eqref{es-6} then implies
\begin{align}\label{es-7}
  \limsup_{n\rightarrow \infty}{E^n} \leq \frac{1}{2\beta c_l}  \big(\frac{Re}{c_0}||\mathbf{F}||_{L^\infty(0,\infty;L^2(\Omega))}^2+\gamma\big):=R_0^2.
\end{align}
\begin{cor}
    For $k\leq 1$ and $||\mathbf{F}||_{L^\infty(0,\infty;L^2(\Omega))} \leq C$,  there exists an absorbing ball of radius $R_0$ to  the scheme \eqref{DDSe-NS}	--\eqref{DSAV-NS}. In particular, if $F$ is independent of time, the scheme \eqref{DDSe-NS}--\eqref{DSAV-NS} generates a discrete dissipative  dynamical system with an absorbing ball on the product space $(\mathbb{H}\times \mathbb{R})^2$.  
\end{cor}

\begin{rem}
    One solves Eq. \eqref{E-NSe1} exactly to find
    \begin{align*}
        q(t)=1+\int_{0}^t e^{-\gamma(t-s)} \int_{\Omega} (\ub \cdot \nabla) \ub \cdot \ub\, dxds. 
    \end{align*}
\end{rem}
\noindent In the discrete case \eqref{DSAV-NS}, a larger $\gamma$ would reduce the error associated with the discretization of the advection term. In the limit $\gamma \rightarrow \infty$, $q\equiv1$, which would eliminate the effect of $q$ and the scheme would be conditionally stable. In the numerical tests we choose $\gamma=1000$.	

 \begin{rem}
     It is clear that the FSAV-BDF2 scheme is applicable to a general class of forced dissipative dynamical system identified in \cite{WaYu2024}:
     \begin{align*}
	&\frac{d\ub}{dt}+A\ub +N(\ub, \ub)=\mathbf{F}(t), \quad \ub|_{t=0}=\ub_0,
\end{align*}
 where $A$ is a positive symmetric linear operator, $N$ is nonlinear and skew-symmetric,  $N(\ub, \ub) \cdot \ub=0$, and $=\mathbf{F} \in \mathbf{L}^\infty(0,\infty;\mathbf{L}^2(\Omega))$. The FSAV-BDF2 scheme then reads
 \begin{align}
	&\delta_t \ub^{n+1} +A \ub^{n+1}+q^{n+1} N\big(2\ub^{n}-\ub^{n-1}, 2\ub^{n}-\ub^{n-1}\big) =\mathbf{F}^{n+1}, \label{DDSe}\\
	&\delta_t q^{n+1}+\gamma q^{n+1}-N\big(2\ub^{n}-\ub^{n-1}, 2\ub^{n}-\ub^{n-1}\big) \cdot \ub^{n+1}=\gamma. \label{DSAVe}
\end{align}	
Following the proof of Theorem \ref{theo1} one can show that the scheme \eqref{DDSe}--\eqref{DSAVe} is unconditionally long-time stable. 
 \end{rem}

\subsection{The scheme for the 2D Navier-Stokes equations}\label{sec(3)}
It is known that for  2D Navier-Stokes equations the $H^1$ norm of the velocity field   is also uniformly bounded in time if the initial data is in the same class. Here we design a BDF2-SAV method for the 2D Navier-Stokes equations in the streamline-vorticity formulation that preserves this important property. Specifically we consider the following boundary conditions
\begin{align}\label{SV-BC}
   \ub \cdot \mathbf{n}=0, \quad \omega:=\nabla \times \ub=0, \text{ on } \partial \Omega. 
\end{align}
In the case of a simply connected domain, the 2D Navier-Stokes equations can then be written in the streamfunction-vorticity form as follows
\begin{align}
  &\frac{\partial \omega}{\partial t}-\frac{1}{Re} \Delta \omega  +\nabla^{\perp} \psi \cdot \nabla \omega=F, \quad \omega=-\Delta \psi, \label{SVe1}\\
  &\omega=\psi=0, \text{ on } \partial \Omega. \label{SVe2}
\end{align}

The BDF2-SAV scheme for the system \eqref{SVe1}--\eqref{SVe2} is :
\begin{align}
    &\delta_t \omega^{n+1}-\frac{1}{Re} \Delta \omega^{n+1}+ q^{n+1} \nabla^{\perp} \overline{\psi}^{n+1} \cdot \nabla \overline{\omega}^{n+1}=F^{n+1}, \quad \omega^{n+1}=-\Delta \psi^{n+1}, \label{DS-SVe1} \\
    &\delta_t q^{n+1}+\gamma q^{n+1}-\int_\Omega  \nabla^{\perp} \overline{\psi}^{n+1} \cdot \nabla \overline{\omega}^{n+1} \omega^{n+1}\, dx=\gamma. \label{DS-SVe2}
\end{align}

\begin{theorem}\label{theo2}
   Suppose $\omega_0 \in H^1_0(\Omega), F\in L^\infty(0, \infty; L^2(\Omega))$, and the scheme \eqref{DS-SVe1}--\eqref{DS-SVe2} is initiated by the first order counterpart. Then for $k\leq 1$,  the scheme \eqref{DS-SVe1}--\eqref{DS-SVe2} is long-time stable in the sense that
    $||\omega^n||_{H^1}+||\psi^n||_{H^3} \leq C$ for any $n\geq 2$ where  the constant $C$ depends on $||\omega_0||_{H^1(\Omega)}$ and $||F||_{L^\infty(0, \infty;L^2(\Omega))}$.
\end{theorem}

\begin{proof}
    The proof for $L^2$ norm is the same as in Theorem \ref{theo1}. For the $H^1$ estimate, we test Eq. \eqref{DS-SVe1} by $-k \Delta \omega$ and perform integration by parts
 \begin{align}\label{H1es-1}
		&|| \nabla V^{n+1}||_G^2-|| \nabla V^{n}||_G^2+\frac{1}{4}||\nabla (\omega^{n+1}-2\omega^n+\omega^{n-1})||^2 +\frac{k}{Re}
		 || \Delta \omega^{n+1}||^2 \nonumber \\
   &\leq \frac{k}{4Re}  ||\Delta  \omega^{n+1}||^2+k Re||\mathbf{F}||^2+k q^{n+1} \int_\Omega \nabla^{\perp} \overline{\psi}^{n+1} \cdot \nabla \overline{\omega}^{n+1} \Delta \omega^{n+1}\, dx.
	\end{align}
 One controls the trilinear term as follows
\begin{align}\label{H1es-2}
    &\Big|q^{n+1} \int_\Omega \nabla^{\perp} \overline{\psi}^{n+1} \cdot \nabla \overline{\omega}^{n+1} \omega^{n+1}\, dx\Big|\leq |q^{n+1}| ||\Delta \omega^{n+1}|| \nabla^{\perp} \overline{\psi}^{n+1}||_{L^4} || \nabla \overline{\omega}^{n+1} ||_{L^4} \nonumber \\
    &\leq C |q^{n+1}| \cdot  ||\Delta \omega^{n+1}|| (||\omega^n||+||\omega^{n-1}||)||\nabla (2\omega^n-\omega^{n-1}) ||^{\frac{1}{2}} ||\Delta (2\omega^n-\omega^{n-1})||^{\frac{1}{2}} \nonumber \\
    &\leq C |q^{n+1}| \cdot ||\Delta \omega^{n+1}|| (||\omega^n||+||\omega^{n-1}||)||(2\omega^n-\omega^{n-1}) ||^{\frac{1}{4}} ||\Delta (2\omega^n-\omega^{n-1})||^{\frac{3}{4}} \nonumber \\
    & \leq C ||\Delta \omega^{n+1}|| (||\omega^n||+||\omega^{n-1}||)^{\frac{5}{4}} (||\Delta \omega^n||+||\Delta \omega^{n-1}||)^{\frac{3}{4}} \nonumber \\
    &\leq \frac{1}{4 Re} ||\Delta \omega^{n+1}||^2+ \epsilon ||\Delta \omega^n||^2+\epsilon ||\Delta \omega^{n-1}||^2+ C(\epsilon) |q^{n+1}|^8 (||\omega^n||+|| \omega^{n-1}||)^{10}.
\end{align}
It follows that
\begin{align}\label{H1es-3}
&|| \nabla V^{n+1}||_G^2-|| \nabla V^{n}||_G^2+\frac{1}{4}||\nabla (\omega^{n+1}-2\omega^n+\omega^{n-1})||^2 +\frac{k}{2Re}
		 || \Delta \omega^{n+1}||^2 \nonumber \\
& \leq k \epsilon (||\Delta \omega^n||^2+||\Delta \omega^{n-1}||^2) +k \Big(C(\epsilon)|q^{n+1}|^8 (||\omega^n||+|| \omega^{n-1}||)^{10}+ Re||\mathbf{F}||^2\Big).
\end{align}
Since there exists a constant $\Lambda_1>0$ such that $||\Delta \omega||^2 \geq \Lambda_1 ||\nabla \omega||^2, \forall \omega \in H^2(\Omega) \cap H^1_0(\Omega)$, one has for any $\alpha \in (0, \frac{1}{2})$
\begin{align}\label{H1es-4}
    \frac{k}{2Re}
		 || \Delta \omega^{n+1}||^2 \geq \frac{k}{Re} (\frac{1}{2}-\alpha) || \Delta \omega^{n+1}||^2+ \frac{k}{Re} \alpha \Lambda_1 || \nabla \omega^{n+1}||^2,
\end{align}
hence
 \begin{align}\label{H1es-5}
&|| \nabla V^{n+1}||_G^2+\frac{k}{Re} \alpha \Lambda_1 || \nabla \omega^{n+1}||^2+k \theta_1 || \nabla \omega^{n}||^2+\frac{k}{Re} (\frac{1}{2}-\alpha) || \Delta \omega^{n+1}||^2 +\theta_2 k ||\Delta \omega^{n}||^2 \nonumber \\
& \leq || \nabla V^{n}||_G^2  + k \theta_1 || \nabla \omega^{n}||^2+ k (\epsilon+\theta_2) ||\Delta \omega^n||^2+\epsilon k||\Delta \omega^{n-1}||^2 \nonumber \\
&\quad +k \Big(C(\epsilon)|q^{n+1}|^8 (||\omega^n||+|| \omega^{n-1}||)^{10}+ Re||\mathbf{F}||^2\Big),
\end{align}
for any $\theta_1, \theta_2>0$. 

One now defines
\begin{align}
    \label{En-dis}
    E^n:=|| \nabla V^{n}||_G^2  + k \theta_1 || \nabla \omega^{n}||^2+ k \theta_3 || \nabla \omega^{n-1}||^2+ k (\epsilon+\theta_2) ||\Delta \omega^n||^2+\epsilon k||\Delta \omega^{n-1}||^2,
\end{align}
with $\theta_3>0$ to be determined.
Provided that $k\leq 1$ and
\begin{align}
    & \frac{\beta}{c_l}+(1+\beta)\theta_1 \leq \frac{\alpha \Lambda_1}{Re}, \label{Cons-1}\\
    & \frac{\beta}{c_l}+(1+\beta)\theta_3 \leq \theta_1, \label{Cons-2}
\end{align}
one obtains
\begin{align}
    &\frac{k}{Re} \alpha \Lambda_1 || \nabla \omega^{n+1}||^2+k \theta_1 || \nabla \omega^{n}||^2 \nonumber\\
    &\geq \beta k || \nabla V^{n+1}||^2 +(1+\beta k)\Big(\theta_1 k ||\nabla \omega^{n+1}||^2+\theta_3 k ||\nabla \omega^n||^2\Big). \label{Temp-1}
\end{align}
Furthermore, if
\begin{align}
    & (1+\beta)(\epsilon+\theta_2) \leq \frac{1}{Re}(\frac{1}{2}-\alpha), \label{Cons-3}\\
    & (1+\beta)\epsilon \leq \theta_2, \label{Cons-4}
\end{align}
then the inequality \eqref{H1es-5} becomes
\begin{align}\label{H1es-6}
(1+\beta k) E^{n+1} \leq E^n+ k \Big(C(\epsilon)|q^{n+1}|^8 (||\omega^n||+|| \omega^{n-1}||)^{10}+ Re||\mathbf{F}||^2\Big).
\end{align}
The constraints \eqref{Cons-1}, \eqref{Cons-2}, \eqref{Cons-3} and \eqref{Cons-4} can always be fulfilled since there are six free parameters. 
By the uniform in time bounds in the $L^2$ norm, it follows $E^n \leq C, n=2, 3 \cdots$. The elliptic regularity theory further implies $||\psi^{n}||_{H^3} \leq C$. This completes the proof.
\end{proof}

\begin{rem}
    The doubly periodic case can be handled in exactly the same way after we separate the average of the vorticity field. Notice that the mean of the vorticity is an invariant of the dynamics assuming the external forcing is mean-zero. The average of the vorticity contains a linear growth term if the average of the external forcing term is non-zero Once we focus on the fluctuation (mean-zero) part, we can apply the Poincar\'e -Wirttinger inequality instead of the classical Poincar\'e inquality. 
\end{rem}

\section{Error analysis}
One performs the error analysis of the general scheme \eqref{DDSe-NS}--\eqref{DSAV-NS} for a fixed fixed final time $T>0$. It should be noted that  long-time convergence of the numerical solutions is generally not expected due to the chaotic and turbulent behavior of the NSE.  For brevity, here one only considers the temporal discretization and focuses on the analysis of the scheme in 2D. The corresponding result in 3D would be local-in-time. The analysis is mostly standard, cf. \cite{LSL2021}.

Recall the vector space $\mathbf{V}:=\{\vb\in \mathbf{H}^1_0(\Omega), \nabla \cdot \vb=0\}$.
For simplicity we suppress the spatial variables in the notation of $\ub$ and $p$, that is, $\ub(t^n)=\ub(t^n, \cdot)$ and $p(t^n)=p(t^n, \cdot)$. We emphasize here that $q(t)\equiv1$.
Denote the error functions by
\begin{align}\label{err-no}
    \eu^n:= \ub^n-\ub(t^n); \quad \ep^n:=p^n- p(t^n); \quad \eq^n:=q^n- q(t^n)=q^n-1.
\end{align} 
We also follow the convention for vectors utilized in G-norm
\begin{align*}
  \eV^n:=\big[\eu^n, \eu^{n-1}\big]^T; \quad   \eQ^n:=\big[\eq^n, \eq^{n-1}\big]^T.
\end{align*}
Then the error functions satisfy the following equations
\begin{align}
	&\delta_t \eu^{n+1} +\frac{1}{Re} \Delta   \eu^{n+1}+\nabla \ep^{n+1}=\ub(t^{n+1})\cdot \nabla\ub(t^{n+1}) -q^{n+1} \overline{\ub}^{n+1}\cdot \nabla  \overline{\ub}^{n+1}+R_{\ub}^{n+1}, \quad \nabla \cdot \eu^{n+1}=0, \label{DDSe-NSer}\\
	&\delta_t \eq^{n+1}+\gamma \eq^{n+1}=\int_\Omega  \overline{\ub}^{n+1}\cdot \nabla  \overline{\ub}^{n+1} \cdot  \ub^{n+1}\, dx-\int_\Omega  \ub(t^{n+1})\cdot \nabla\ub(t^{n+1}) \cdot  \ub(t^{n+1})\, dx, \label{DSAV-NSer}
\end{align}	
	where $R_{\ub}^{n+1}:=\delta_t \ub^{n+1} (\cdot )-\frac{\partial \ub}{\partial t}(t^{n+1}, \cdot)$.

 \begin{theorem}
     For the 2D NSE, assuming $\ub \in {H}^3(0,T; \mathbf{H}^1_0(\Omega)) \cap {H}^1(0,T; \mathbf{H}^2(\Omega)) $, and $p \in L^2(0,T; H^1(\Omega)) $, then there holds
     \begin{align}
         ||\eV^{m+1}||_G^2+ ||\eQ^{m+1}||_G^2+k \sum_{n=1}^m ||\nabla \eu^n||^2 +k \gamma \sum_{n=1}^m | \eq^n|^2 +k \sum_{n=1}^m ||\ep^{n+1}||_{L^2(\Omega)/R}^2  \leq C k^4, \quad 1\leq m\leq N-1.
     \end{align}
 \end{theorem}

 \begin{proof}
Taking the $\bf{L}^2$ inner product of Eqs. \eqref{DDSe-NSer} with $k\eu^{n+1}$ and performing integration by parts, one obtains
\begin{align}\label{err-1}
    ||\eV^{n+1}||_G^2-||\eV^{n}||_G^2+\frac{k}{Re}||\nabla \eu^{n+1}||^2 \leq k\Big(\ub(t^{n+1})\cdot \nabla\ub(t^{n+1}) -q^{n+1} \overline{\ub}^{n+1}\cdot \nabla  \overline{\ub}^{n+1}, \eu^{n+1}\Big)+k||R_{\ub}^{n+1}||\cdot ||\eu^{n+1}||.
\end{align}
Multiplying Eq. \eqref{DSAV-NSer} by $k\eq^{n+1}$ yields
\begin{align}\label{err-2}
|\eQ^{n+1}|_G^2-|\eQ^{n}|_G^2+k\gamma|\eq^{n+1}|^2\leq k\eq^{n+1}\left\{\int_\Omega  \big(\overline{\ub}^{n+1}\cdot \nabla \big) \overline{\ub}^{n+1} \cdot  \ub^{n+1}\, dx-\int_\Omega  \big({\ub}^{n+1}\cdot \nabla \big) {\ub}^{n+1} \cdot  \ub^{n+1}\, dx\right\}.
\end{align}
In light of the fact $q(t)=1$, one writes the first term on the right hand side of Eq. \eqref{err-1} as following
\begin{align}\label{non-1}
    &\Big(q(t^{n+1})\ub(t^{n+1})\cdot \nabla\ub(t^{n+1}) -q^{n+1} \overline{\ub}^{n+1}\cdot \nabla  \overline{\ub}^{n+1}, \eu^{n+1}\Big) \nonumber \\
    &= \Big([\ub(t^{n+1})-\overline{\ub}^{n+1}]\cdot \nabla\ub(t^{n+1}), \eu^{n+1}\Big)+\Big(\overline{\ub}^{n+1}\cdot \nabla  [\ub(t^{n+1})-\overline{\ub}^{n+1}], \eu^{n+1}\Big)\nonumber \\
    &\quad -\eq^{n+1}\Big( \overline{\ub}^{n+1}\cdot \nabla  \overline{\ub}^{n+1}, \eu^{n+1}\Big).
\end{align}
Likewise the right hand side of Eq. \eqref{err-2} can be written as
\begin{align}\label{non-2}
    &\eq^{n+1}\Big( \overline{\ub}^{n+1}\cdot \nabla  \overline{\ub}^{n+1}, \eu^{n+1}\Big)-\eq^{n+1} \Big(\overline{\ub}^{n+1}\cdot \nabla  [\ub(t^{n+1}) \nonumber  -\overline{\ub}^{n+1}], \ub(t^{n+1})\Big)\\
    &-\eq^{n+1}
     \Big([\ub(t^{n+1})-\overline{\ub}^{n+1}]\cdot \nabla(\ub(t^{n+1}),  \ub(t^{n+1})\Big)
\end{align}
Taking summation of the inequalities \eqref{err-1} and \eqref{err-2}, in view of \eqref{non-1} and \eqref{non-2} one derives
\begin{align}
    \label{err-3}
    &(||\eV^{n+1}||_G^2+|\eQ^{n+1}|_G^2)-(||\eV^{n}||_G^2+|\eQ^{n}|_G^2)+\frac{k}{Re}||\nabla \eu^{n+1}||^2+k\gamma|\eq^{n+1}|^2 \leq k||R_{\ub}^{n+1}||\cdot ||\eu^{n+1}|| \nonumber \\
    &+k \Big([\ub(t^{n+1})-\overline{\ub}^{n+1}]\cdot \nabla\ub(t^{n+1}), \eu^{n+1}\Big)+k \Big(\overline{\ub}^{n+1}\cdot \nabla  [\ub(t^{n+1})-\overline{\ub}^{n+1}], \eu^{n+1}\Big)\nonumber \\
    &-k\eq^{n+1} \Big(\overline{\ub}^{n+1}\cdot \nabla  [\ub(t^{n+1})   -\overline{\ub}^{n+1}], \ub(t^{n+1})\Big)-k\eq^{n+1}
     \Big([\ub(t^{n+1})-\overline{\ub}^{n+1}]\cdot \nabla\ub(t^{n+1}),  \ub(t^{n+1})\Big)  \\
     &:=\sum_{i=1}^5 I_i,  \nonumber 
\end{align}
where $I_i, i=1\ldots 5$ are the five terms on the right hand side of \eqref{err-3}.

For $I_1$,
\begin{align}
    \label{I1}
    k||R_{\ub}^{n+1}||\cdot ||\eu^{n+1}|| \leq \epsilon \frac{k}{Re} ||\nabla \eu^{n+1}||^2+C(\epsilon) k Re ||R_{\ub}^{n+1}||^2,
\end{align}
with $\epsilon>0$ to be determined. By virtue of the notation in \eqref{err-no} one writes
\begin{align} \label{dif}
    \ub(t^{n+1})-\overline{\ub}^{n+1}=D\ub(t^{n+1})-2\eu^n+\eu^{n-1}
\end{align}
with $D\ub(t^{n+1}):=\ub(t^{n+1})-2\ub(t^{n})+\ub(t^{n-1})$. Then 
$I_2$ is estimated  by H\"{o}lder's inequality and the Sobolev embedding as follows
\begin{align}
    \label{I2}
    |I_2| &\leq Ck ||\ub(t^{n+1})-\overline{\ub}^{n+1}|| \cdot ||\ub(t^{n+1})||_{H^2}\cdot ||\nabla \eu^{n+1}|| \nonumber \\
    &\leq \epsilon \frac{k}{Re}||\nabla \eu^{n+1}||^2+k C(\epsilon)Re ||\ub(t^{n+1})||_{H^2}^2 (||D\ub(t^{n+1})||^2+||2\eu^n-\eu^{n-1}||^2) \nonumber \\
    &\leq \epsilon \frac{k}{Re}||\nabla \eu^{n+1}||^2+k C(\epsilon)Re ||\ub(t^{n+1})||_{H^2}^2 (||D\ub(t^{n+1})||^2+||\eV^n||_G^2).
\end{align}
By \eqref{dif} one writes $I_3$ as
\begin{align}\label{I3_1}
I_3&=k \Big(\overline{\ub}^{n+1}\cdot \nabla  D\ub(t^{n+1}), \eu^{n+1}\Big)- k \Big([2\eu^n-\eu^{n-1}]\cdot \nabla  [2\eu^n-\eu^{n-1}], \eu^{n+1}\Big)\nonumber \\
&-k \Big([2\ub(t^n)-\ub(t^{n-1})]\cdot \nabla  [2\eu^n-\eu^{n-1}], \eu^{n+1}\Big).
\end{align}
 Recall the 2D Ladyzhenskaya's inequality \cite{Ladyzhenskaya1969}
\begin{align}
    \label{Lady}
    ||\ub||_{L^4}^2 \leq 2 ||\ub||\cdot ||\nabla \ub||.
\end{align}
The first term in \eqref{I3_1} satisfies
\begin{align*}
       \Big|  k \Big(\overline{\ub}^{n+1}\cdot \nabla  D\ub(t^{n+1}), \eu^{n+1}\Big)    \Big|  &=\Big|  k \Big(\overline{\ub}^{n+1}\cdot \nabla \eu^{n+1} , D\ub(t^{n+1})\Big)    \Big| \nonumber \\
       &\leq k ||\Big(\overline{\ub}^{n+1}||_{L^4} ||\nabla \eu^{n+1}||\cdot ||D\ub(t^{n+1})||_{L^4} \\
       & \leq C k||\overline{\ub}^{n+1}||^{1/2} ||\nabla \overline{\ub}^{n+1}||^{1/2} ||\nabla \eu^{n+1}||\cdot ||D\ub(t^{n+1})||^{1/2}  ||\nabla D\ub(t^{n+1})||^{1/2} \\
       &\leq \epsilon \frac{k}{Re} ||\nabla \eu^{n+1} ||^2+ C(\epsilon)k Re ||\overline{\ub}^{n+1}||^2 ||\nabla D\ub(t^{n+1})||^2+ Ck ||\nabla \overline{\ub}^{n+1}||^2||D\ub(t^{n+1})||^2.
\end{align*}
Likewise the second term in \eqref{I3_1} has the estimate
\begin{align*}
    \Big|k \Big([2\eu^n-\eu^{n-1}]\cdot \nabla  [2\eu^n-\eu^{n-1}], \eu^{n+1}\Big)\Big| 
    &\leq Ck ||2\eu^n-\eu^{n-1}|| \cdot ||\nabla  [2\eu^n-\eu^{n-1}]||\cdot ||\nabla \eu^{n+1} || \\
    &\leq \epsilon \frac{k}{Re} ||\nabla \eu^{n+1} ||^2+ C(\epsilon)k Re ||\nabla  [2\eu^n-\eu^{n-1}]||^2||\eV^n||_G^2.
\end{align*}
The third term in $I_3$ is estimated in the same way as \eqref{I2}.
 Hence
\begin{align}\label{I3}
    &|I_3| \leq 3 \epsilon \frac{k}{Re} ||\nabla \eu^{n+1} ||^2+ C(\epsilon)k Re ||\overline{\ub}^{n+1}||^2||\nabla D\ub(t^{n+1})||^2 +Ck ||\nabla \overline{\ub}^{n+1}||^2||D\ub(t^{n+1})||^2\nonumber \\
    &\quad \quad + C(\epsilon)k Re (||\nabla  [2\eu^n-\eu^{n-1}]||^2+||2\ub(t^n)-\ub(t^{n-1})||_{H^2}^2)||\eV^n||_G^2.
\end{align}
One controls $I_4$ as follows
\begin{align}
    \label{I4}
   |I_4|  &\leq k |\eq^{n+1}|\left\{ \Big|\Big(\overline{\ub}^{n+1}\cdot \nabla  D \ub(t^{n+1}), \ub(t^{n+1})\Big)\Big|+\Big|\Big(\overline{\ub}^{n+1}\cdot \nabla  (2\eu^n-\eu^{n-1}), \ub(t^{n+1})\Big)\Big|\right\} \nonumber \\
    &\leq \theta k \gamma |\eq^{n+1}|^2+k\frac{C(\theta)}{ \gamma} ||\overline{\ub}^{n+1}||^2 ||\ub(t^{n+1})||_{H^2}^2 ||\nabla  D \ub(t^{n+1})||^2+\tau k |\eQ^{n+1}|_G^2 ||\nabla \overline{\ub}^{n+1}||^2 \nonumber \\
    &\quad + C(\tau)k ||\ub(t^{n+1})||_{H^2}^2||\eV^n||_{G}^2,
\end{align}
with $\theta, \tau>0$ to be determined.
Likewise $I_5$ satisfies the bound
\begin{align}
    \label{I5}
    |I_5|&\leq k |\eq^{n+1}|\left\{ \Big| \Big(D\ub(t^{n+1})\cdot \nabla\ub(t^{n+1}),  \ub(t^{n+1})\Big)  \Big|+ \Big|\Big([2\eu^n-\eu^{n-1}]\cdot \nabla\ub(t^{n+1}),  \ub(t^{n+1})\Big)\Big|\right\} \nonumber \\
    &\leq 2\theta k \gamma |\eq^{n+1}|^2+k\frac{C(\theta)}{ \gamma}||\nabla \ub(t^{n+1})||^2||\ub(t^{n+1})||_{H^2}^2( ||D\ub(t^{n+1})||^2+||\eV^n||_{G}^2).
\end{align}

 Taking $\epsilon=\frac{1}{10}$ and $\theta=\frac{1}{6}$, and collecting the estimates of $I_i, i=1\ldots 5$ into \eqref{err-3}, one obtains
 \begin{align}
     \label{err-4}
     &(||\eV^{n+1}||_G^2+|\eQ^{n+1}|_G^2)-(||\eV^{n}||_G^2+|\eQ^{n}|_G^2)+\frac{k}{2Re}||\nabla \eu^{n+1}||^2+k\frac{\gamma}{2}|\eq^{n+1}|^2 \nonumber \\
     &\leq  Ck  \left\{\big(Re+C(\tau)+\frac{1}{\gamma}||\nabla \ub(t^{n+1})||^2\big)||\ub(t^{n+1})||_{H^2}^2 +||2\ub(t^n)-\ub(t^{n-1})||_{H^2}^2+||\nabla  [2\eu^n-\eu^{n-1}]||^2\right\}||\eV^n||_G^2 \nonumber \\
     &\quad +Ck  \left\{\big(Re+\frac{1}{\gamma}||\overline{\ub}^{n+1}||^2+\frac{1}{\gamma}||\nabla \ub(t^{n+1})||^2\big)||\ub(t^{n+1})||_{H^2}^2+Re||\overline{\ub}^{n+1}||^2\right\} ||\nabla D\ub(t^{n+1})||^2 \nonumber \\
     &\quad +C k Re ||R_{\ub}^{n+1}||^2+\tau k |\eQ^{n+1}|_G^2 ||\nabla \overline{\ub}^{n+1}||^2+Ck ||\nabla \overline{\ub}^{n+1}||^2||D\ub(t^{n+1})||^2.
 \end{align}
 Taking summation of \eqref{err-4} from $n=1$ to $m\leq N-1$,  in light of the stability estimates of the numerical solution as well as the regularity assumption on the exact solution, one derives
 \begin{align}
     \label{err-5}
     &(||\eV^{m+1}||_G^2+|\eQ^{n+1}|_G^2)+\frac{k}{2Re}\sum_{n=1}^m||\nabla \eu^{n+1}||^2+k\frac{\gamma}{2}\sum_{n=1}^m|\eq^{n+1}|^2 \nonumber \\
     &\leq C(Re, \tau, \gamma)k\sum_{n=1}^m \left\{1+||\nabla  [2\eu^n-\eu^{n-1}]||^2\right\}||\eV^n||_G^2 +\tau k\sum_{n=2}^m||\nabla \overline{\ub}^{n}||^2 |\eQ^{n}|_G^2 +\tau k |\eQ^{m+1}|_G^2 ||\nabla \overline{\ub}^{m+1}||^2  \nonumber \\
     &\quad +C(Re, \gamma)k\sum_{n=1}^m ||\nabla D\ub(t^{n+1})||^2 +C\max_{n\geq 1}||D\ub(t^{n+1})||^2 \sum_{n=1}^m k||\nabla \overline{\ub}^{n+1}||^2+ Re \sum_{n=1}^m ||R_{\ub}^{n+1}||^2+||\eV^{1}||_G^2.
 \end{align}
Since $k ||\nabla \overline{\ub}^{m+1}||^2 \leq C$, one chooses the largest $\tau$ such that $\tau k ||\nabla \overline{\ub}^{m+1}||^2 \leq \frac{1}{2}$. The inequality \eqref{err-5} then becomes
\begin{align}
    \label{err-6}
    &(||\eV^{m+1}||_G^2+|\eQ^{n+1}|_G^2)+\frac{k}{Re}\sum_{n=1}^m||\nabla \eu^{n+1}||^2+k \gamma\sum_{n=1}^m|\eq^{n+1}|^2 \nonumber \\
     &\leq C(Re,  \gamma)k\sum_{n=1}^m \left\{1+||\nabla  [2\eu^n-\eu^{n-1}]||^2\right\}||\eV^n||_G^2 +\tau k\sum_{n=2}^m||\nabla \overline{\ub}^{n}||^2 |\eQ^{n}|_G^2 +Ck^4.
\end{align}
 An application of the discrete Gronwall's inequality implies
 \begin{align}
  \label{err-7}
     (||\eV^{m+1}||_G^2+|\eQ^{n+1}|_G^2)+\frac{k}{Re}\sum_{n=1}^m||\nabla \eu^{n+1}||^2+k \gamma\sum_{n=1}^m|\eq^{n+1}|^2 \leq C(Re,  \gamma)k^4.
 \end{align}

The error estimate for pressure is standard, cf. \cite{GiRa1986}.  Since the divergence operator is an isomorphism  from $\mathbf{V}^{\perp}$ onto $L^2(\Omega)/R$, there exist $\vb^{n+1} \in \mathbf{H}_0^1(\Omega)$ such that
\begin{align*}
    -\nabla \cdot \vb^{n+1}= \ep^{n+1}, \quad ||\nabla \vb^{n+1}|| \leq C ||\ep^{n+1}||_{L^2(\Omega)/R}.
\end{align*}
Then it follows from the error equations \eqref{DDSe-NSer} that
\begin{align}
    \label{err-p1}
    &||\ep^{n+1}||_{L^2(\Omega)/R}^2=\big(\nabla \ep^{n+1}, \vb^{n+1}\big) \nonumber \\
    &\leq ||\nabla \eu^{n+1}|| \cdot ||\nabla \vb^{n+1}||+ ||R_{\ub}^{n+1}||\cdot || \vb^{n+1}||+\Big|\Big([\ub(t^{n+1})-\overline{\ub}^{n+1}]\cdot \nabla\ub(t^{n+1}), \vb^{n+1}\Big)\Big|\nonumber \\
    &\quad +\Big|\Big(\overline{\ub}^{n+1}\cdot \nabla  [\ub(t^{n+1})-\overline{\ub}^{n+1}], \vb^{n+1}\Big)\Big|
     +\Big|\eq^{n+1}\Big( \overline{\ub}^{n+1}\cdot \nabla  \overline{\ub}^{n+1}, \vb^{n+1}\Big)\Big| \nonumber \\
     &\leq C \big(||\nabla \eu^{n+1}||+||R_{\ub}^{n+1}||\big)||\nabla \vb^{n+1}||+ \big(||\nabla D\ub(t^{n+1})||+||\nabla (2\eu^{n}-\eu^{n-1})||\big)||\nabla \ub(t^{n+1})||\cdot ||\nabla \vb^{n+1}|| \nonumber \\
     &+C||\overline{\ub}^{n+1}||^{1/2}||\nabla \overline{\ub}^{n+1}||^{1/2} ||\nabla \vb^{n+1}|| \big( ||D\ub(t^{n+1})||^{1/2}||\nabla D\ub(t^{n+1})||^{1/2}+||2\eu^{n}-\eu^{n-1}||^{1/2}||\nabla (2\eu^{n}-\eu^{n-1})||^{1/2}\big) \nonumber \\
     &+C|\eq^{n+1}| \cdot ||\overline{\ub}^{n+1}||\cdot ||\nabla \overline{\ub}^{n+1}|| \cdot ||\nabla \vb^{n+1}||,
\end{align}
where integration by parts and Ladyzhenskaya's inequality \eqref{Lady} have been utilized in the estimate of the  trilinear terms. It follows that
\begin{align}
    \label{err-p2}
    &||\ep^{n+1}||_{L^2(\Omega)/R} \leq C \big(||\nabla \eu^{n+1}||+||R_{\ub}^{n+1}||+||\nabla D\ub(t^{n+1})||+||\nabla (2\eu^{n}-\eu^{n-1})||\big)+C|\eq^{n+1}|\cdot ||\nabla \overline{\ub}^{n+1}|| \nonumber \\
    &+C||\nabla \overline{\ub}^{n+1}||  \big( ||D\ub(t^{n+1})||+||2\eu^{n}-\eu^{n-1}||\big)+||\nabla D\ub(t^{n+1})||+||\nabla (2\eu^{n}-\eu^{n-1})||. 
\end{align}
Therefore
\begin{align}
    \label{err-p3}
    &k\sum_{n=1}^m||\ep^{n+1}||^2_{L^2(\Omega)/R} \nonumber \\
    &\leq C k\sum_{n=1}^m\big(||\nabla \eu^{n+1}||^2+||R_{\ub}^{n+1}||^2+||\nabla D\ub(t^{n+1})||^2+||\nabla (2\eu^{n}-\eu^{n-1})||^2\big)+C\max_{n\geq 1} |\eq^{n+1}|^2 \sum_{n=1}^m k||\nabla \overline{\ub}^{n+1}||^2 \nonumber \\
    &+C\sum_{n=1}^m k||\nabla \overline{\ub}^{n+1}||^2 \max_{n\geq 1} \big( ||D\ub(t^{n+1})||^2+||2\eu^{n}-\eu^{n-1}||^2\big)+\sum_{n=1}^m k\big(||\nabla D\ub(t^{n+1})||^2+||\nabla (2\eu^{n}-\eu^{n-1})||^2\big)\nonumber \\
    &\leq C k^4.
\end{align}
This completes the proof.
 \end{proof}
	
	\section{Numerical experiments}\label{s-nu}
	Numerical tests are performed to demonstrate the accuracy, efficiency and robustness of the BDF2-SAV schemes. In all tests, the scheme is initiated by a first order scheme with explicit discretization of the nonlinear term.

	\subsection{ Accuracy} 
 We perform three tests for verification of accuracy using manufactured solutions.  In the first two tests, finite element is used for spatial discretization for the SAV-BDF2 schemes in the primitive form and the streamfunction-vorticiity form, respectively.  The implementation is carried out using the open-source package FreeFem++. In the first two examples, the computational domain is  $\Omega=[0,1]^2$, the final time is set at $T=100$, $h$ and $k$ are simultaneously refined through the relation $h=\frac{1}{2}k$, so that the error is dominated by the temporal error.

 For the SAV-BDF2 algorithm in the primitive form,   The test problem is manufactured according to the exact solution
	\begin{align*}
		&u_1= 2x^2y\sin(t) (x - 1)^2(y - 1)^2 + x^2 y^2 \sin(t) (2y - 2)(x - 1)^2, \\
		&u_2=- 2xy^2\sin(t)(x - 1)^2(y - 1)^2 - x^2y^2\sin(t)(2x - 2)(y - 1)^2,\\
		&p=cos(t) xy.
	\end{align*}
 The parameters are $Re=100$, $\gamma=1000$. The Taylor-Hood $P2-P1$ finite element pair is utilized for spatial discretization.
 The relative error  for    $\ub$ and $p$  in the $L^2$ norm, and relative error for $q$ are displayed in the Table \ref{T_Conv}.

	\begin{table}[hbt]
	\begin{center}
		\caption{Relative error   for    $\ub$ and $p$  in the $L^2$ norm, and  error for $q$  with $k=2 h$ and $P2-P1$ Taylor-Hood finite element.} \label{T_Conv}
		\begin{tabular}{cp{2cm}<{\centering}p{1cm}<{\centering}p{2cm}<{\centering}p{1cm}<{\centering}p{2cm}<{\centering}p{1cm}<{\centering}p{2cm}<{\centering}p{1cm}<{\centering}p{2cm}}
			\hline  $k$  &  $\|e_{u}\|$    &order  &      $\|e_p\|$    &order    &$|q-1|$ \\[0.5ex]\hline
			0.5    &0.0801977   &       &0.0121033   &                                &2.47015e-09\\
			0.25    &0.0108756   &2.88        &0.00302578   &2.00   &5.91561e-11\\
			0.125   &0.00151803   &2.84        &0.000756443  &2.00   &6.28753e-12\\
			0.0625  &0.000242224   &2.65       &0.000189111   &2.00   &1.74971e-13\\
			0.03125   &4.51685e-05   &2.42     &4.72777e-05  &2.00  & 2.22045e-16\\
			0.015625  &  9.5513e-06   &   2.24        & 1.18194e-05  &  2.00        &6.66134e-16                             \\
			0.0078125 & 2.17576e-06 &    2.13       &2.95486e-06    & 2.00     &1.11022e-16
\\			\hline
		\end{tabular}
	\end{center}
\end{table}

For the SAV-BDF2 algorithm in the streamfunction-vorticity formulation,   the true solutions are 
	\begin{align*}
		&\omega= \sin(t) \sin(\pi x) \sin(2\pi y), \\
		&\psi=-10\cos(t)x(x-1)y(y-1),
	\end{align*}
 in which $Re=10$, $\gamma=1000$. $P2$ finite elements are utilized for both variables.
 The relative error  for    $\omega$ and $\psi$  in the $L^2$ norm, and relative error for $q$ are displayed in the Table \ref{T_Conv2}. 
	\begin{table}[hbt]
	\begin{center}
		\caption{Relative error   for    $\omega$ and $\psi$  in the $L^2$ norm, and  error for $q$  with $k=2 h$ and $P2-P2$ finite element pair} \label{T_Conv2}
		\begin{tabular}{cp{2cm}<{\centering}p{1cm}<{\centering}p{2cm}<{\centering}p{1cm}<{\centering}p{2cm}<{\centering}p{1cm}<{\centering}p{2cm}<{\centering}p{1cm}<{\centering}p{2cm}}
			\hline  $k$  &  $\|e_{\omega}\|$    &order  &      $\|e_\psi\|$    &order    &$|q-1|$ \\[0.5ex]\hline
			0.5    &0.501288   &       &0.0110108   &                                &2.47015e-09\\
			0.25    & 0.098924  &2.34        &0.00196049   &2.49   &0.000131017\\
			0.125   & 0.022645  &2.12       &0.000417379  &2.23   &6.81064e-06\\
			0.0625  & 0.00560761  &2.01       &0.000100779   &2.05   &3.72619e-07\\
			0.03125   & 0.00139695  &2.00     &2.49413e-05  &2.01  & 1.0624e-09\\
			0.015625  & 0.000348469  &   2.00        & 6.21157e-06  &  2.01        &6.27696e-11                             \\
			0.0078125 & 8.70057e-05 &    2.00       &1.55032e-06    & 2.00     &3.80607e-12
\\			\hline
		\end{tabular}
	\end{center}
\end{table}	

Since periodic boundary conditions will be used in the following for simulations of coherent structures, the SAV-BDF2 scheme in the stream function formulation is also implemented in Matlab utilizing Fourier collocation methods. The exact solutions are 
\begin{align*}
    &\omega=\sin(t) \sin(2\pi x) \sin(2\pi y), \quad \psi=\cos(t) \cos(2\pi x) \cos(2\pi y).
\end{align*}
The parameters are $Re=10, \gamma=1000, T=100, \Omega=(0, 1)^2$. $32$ Fourier modes are employed. The relative error in $l^\infty$ norm and convergence order is displayed in Table \ref{stream_conv}.

	\begin{table}[hbt]
	\begin{center}
		\caption{Relative error   for    $\omega$ and $\psi$  in the $l^\infty$ norm, and  error for $q$  with  $32$ Fourier modes in each direction} \label{stream_conv}
		\begin{tabular}{cp{2cm}<{\centering}p{1cm}<{\centering}p{2cm}<{\centering}p{1cm}<{\centering}p{2cm}<{\centering}p{1cm}<{\centering}p{2cm}<{\centering}p{1cm}<{\centering}p{2cm}}
			\hline  $k$  &  $\|e_{\omega}\|_{\infty}$    &order  &      $\|e_\psi\|_{\infty}$    &order    &$|q-1|$ \\[0.5ex]\hline
			0.05    &1.36619   &       &1.818292e-03   &                                &6.096786e-03\\
			0.025    & 1.025531e-03  &10.38        &8.159268e-06   &7.80  &3.004133e-09\\
			0.0125   & 2.553669e-04  &2.00      &2.031416e-06  &2.00  &6.844236e-11\\
			0.00625  & 6.363869e-05  &2.00       &5.064025e-07   &2.00   &4.260259e-12\\
			0.003125   & 1.588605e-05  &2.00     &1.264143e-07  &2.00  & 2.654543e-13\\
			0.0015625  & 3.968483e-06  &   2.00        & 3.157964e-08  &  2.00        &1.643130e-14                             \\
			0.00078125 & 9.917372e-07 &    2.00       &7.891879e-09    & 2.00     &1.110223e-15
\\			\hline
		\end{tabular}
	\end{center}
\end{table}

\subsection{The Kolmogorov flow}	
As a Benchmark test, one considers the 2D NSE with the Kolmogorov forcing and periodic boundary conditions in the domain $\Omega=(0, 2\pi)^2$, as is detailed in \cite{Armbruster1996}. A typical Kolmogorov forcing takes the form
\begin{align}\label{Kol}
    \mathbf{f}:=\Big[\frac{m^3}{Re} \cos(my), 0\Big]^T, \quad m \in \mathbb{N}^+.
\end{align}
The steady-state solution, also known as the basic Kolmogorov flow, is given in a streamfunction formulation as $\psi=\sin(my)$.  The steady-state solution is stable when $Re$ is relatively small. It becomes unstable as $Re$ crosses a critical value and coherent structure develops. In the following simulations, the SAV-BDF2 scheme is implemented in Matlab in conjunction with the Fourier collocation method using $256$ modes.

\subsubsection{Long-time stability}
One verifies the long-time stability of the SAV-BDF2 scheme for the 2D NSE  in the stream function formulation with the  Kolmogorov forcing \eqref{Kol} and $m=2, Re=100$. The initial stream function is a (periodic) perturbation of the steady-state solution
\begin{align*}
    \psi(0, x, y)=\sin(2y)+0.001\sin(2 x)\sin(2 y).
\end{align*}
The evolution of the $L^2$ norm and $H^1$ norm of the vorticity by the SAV-BDF2 scheme for $T=1000$ is depicted in Fig. \ref{Sta_plot}  with 256 Fourier modes and $k=0.01, 0.005, 0.0025$ respectively. It is noted that the BDF2 IMEX scheme with explicit discretization of the advection term blows up for $k=0.003$.

\begin{figure}[!ht]
		\subfigure
		{
			\centering
			\begin{minipage}[t]{1.0\linewidth}
				\includegraphics[width=3.4in]{  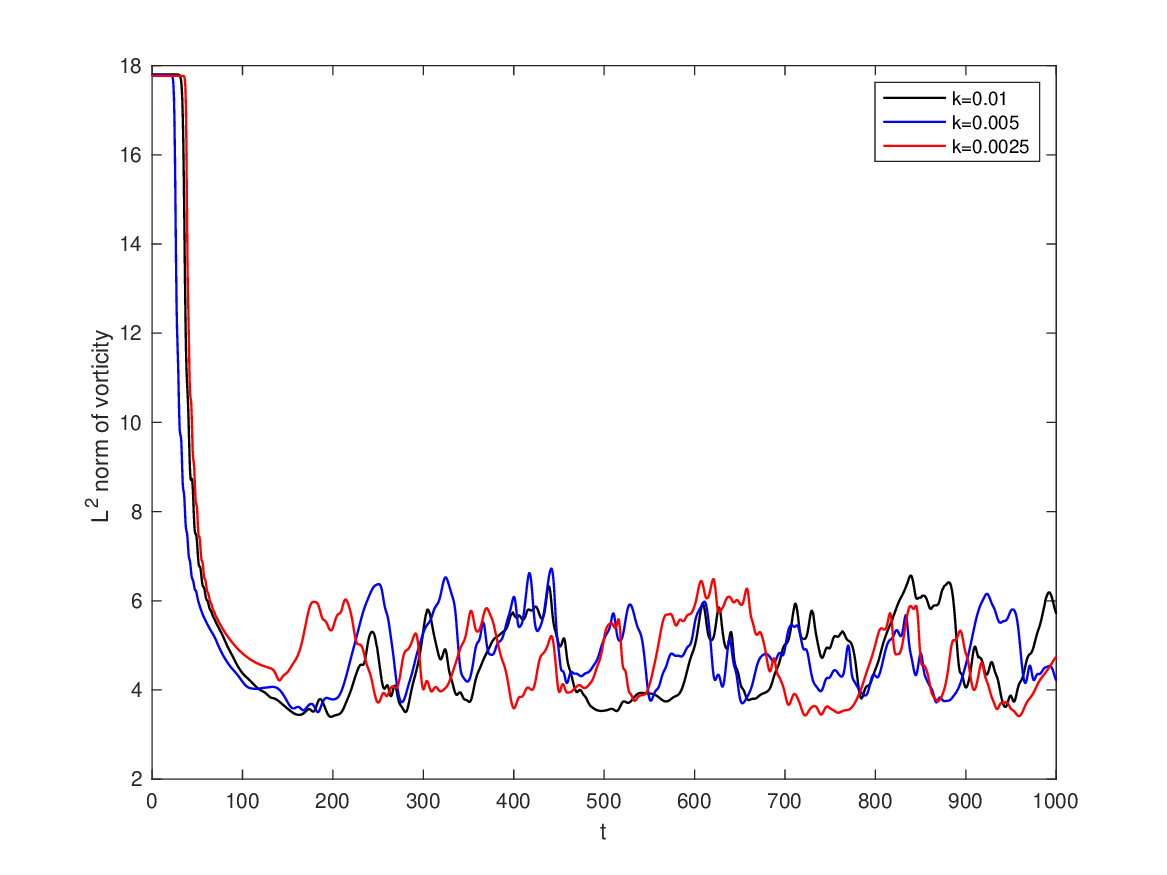}
				\includegraphics[width=3.4in]{  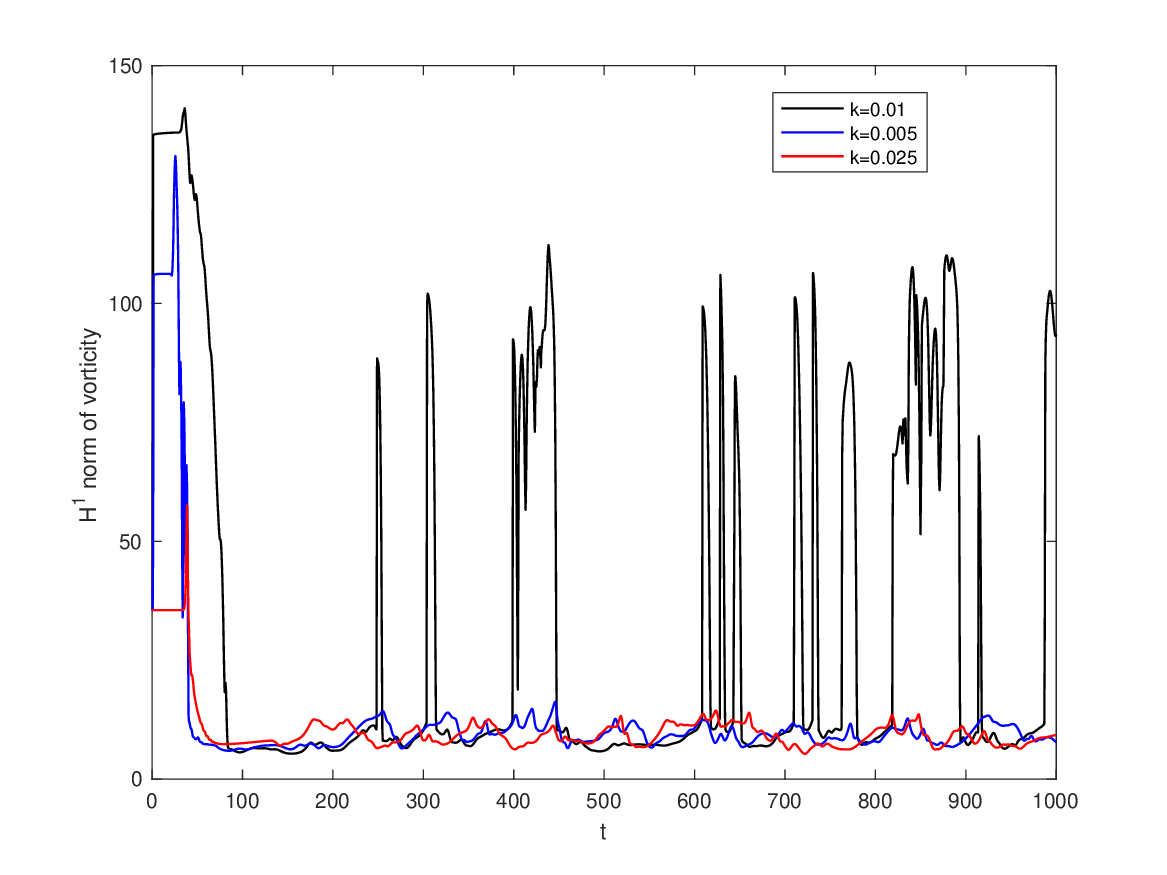}
			\end{minipage}
		}
		\caption{The $L^2$ norm and $H^1$ norm of the vorticity as a function of time by the SAV-BDF2 scheme with 256 Fourier modes and $k=0.01, 0.005, 0.0025$ respectively.}
		\label{Sta_plot}
	\end{figure}

\subsubsection{Bursting}
The initial condition used in this simulation is a perturbation to the basic Kolmogorov flow
\begin{align*}
    \psi(0, x, y)=\sin(2y)+0.001\sin(2\pi x)\sin(2\pi y).
\end{align*}
It is discovered through numerical simulation in \cite{Armbruster1996} that at $Re=25.70$ the solution is quasi-periodic which consists of a traveling structure plus additional time dependent behavior, and  bursts occur intermittently at $Re \geq 25.77$. In the following simulation, the Reynolds number is taken to be $25.7715$, $k=0.001$, and the final time is $T=10000$. Fourier collocation method is utilized for spatial discretization with $256$ Fourier modes.

The real and imaginary part of the Fourier coefficient of mode $e^{iy}$ is plotted in Fig. \ref{mode} as a function of time. Bursts occur intermittently characterized by a sudden dramatic change of magnitude of the Fourier coefficient. The evolution of the maximum vorticity and the $L^2$ norm of the gradient is displayed in Fig. \ref{L2}, which further corroborates the appearance of bursts. The dynamics undergo a long laminar regime, then a chaotic explosion (burst) ensues, followed by another long period of laminar regime. The pattern repeats in time. It is noted in \cite{Armbruster1996} that the occurrences of bursts do not correlate with each other. Fig. \ref{bar} confirms that the time interval between bursts is not constant and appears to be random. The power spectrum density of the fluctuation of the maximum vorticity is shown in Fig. \ref{psd}. The wide spread of frequencies indicate non-periodic motion, while the concentration of power at low frequencies suggests the intermittency of bursting phenomenon.
\begin{figure}[!ht]
	\centering
	\subfigure{
		\includegraphics[width=5in]{  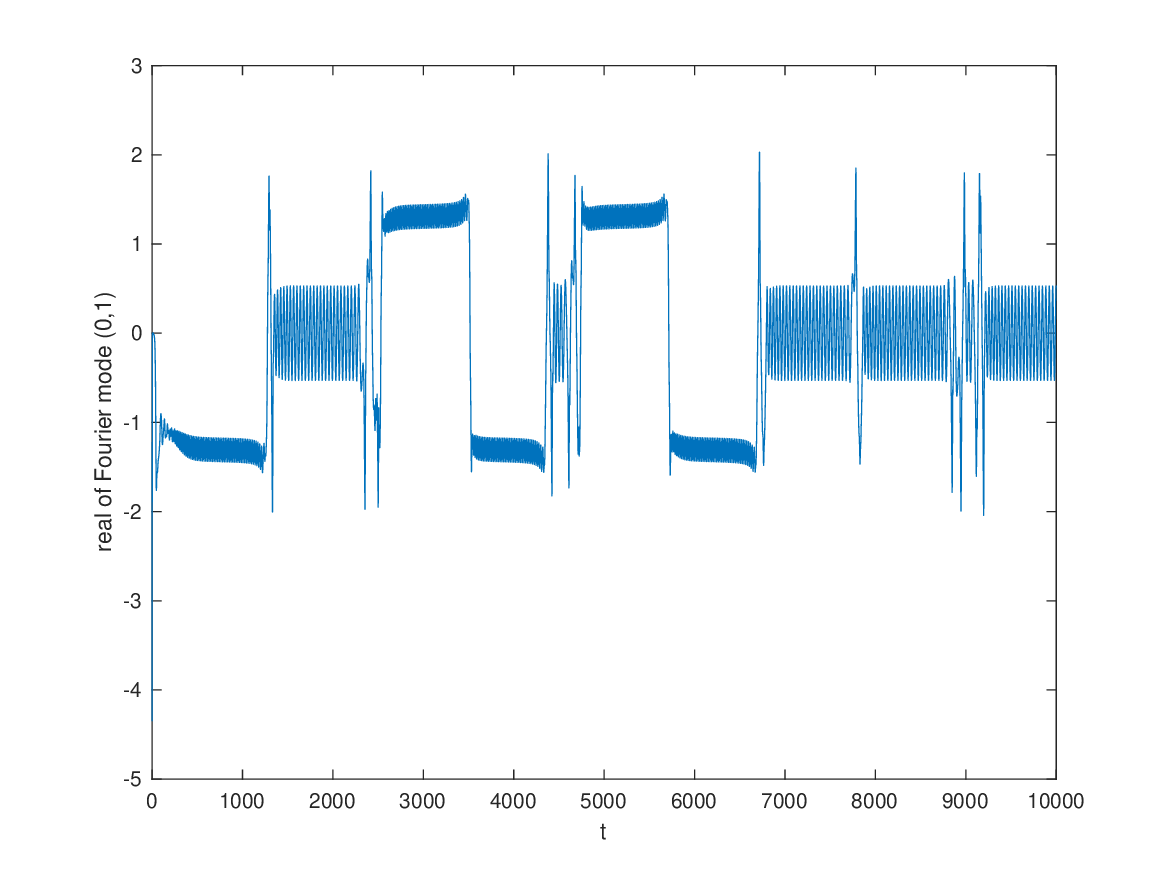}
	}
\hspace{0.5cm}
	\subfigure{
		\includegraphics[width=5in]{  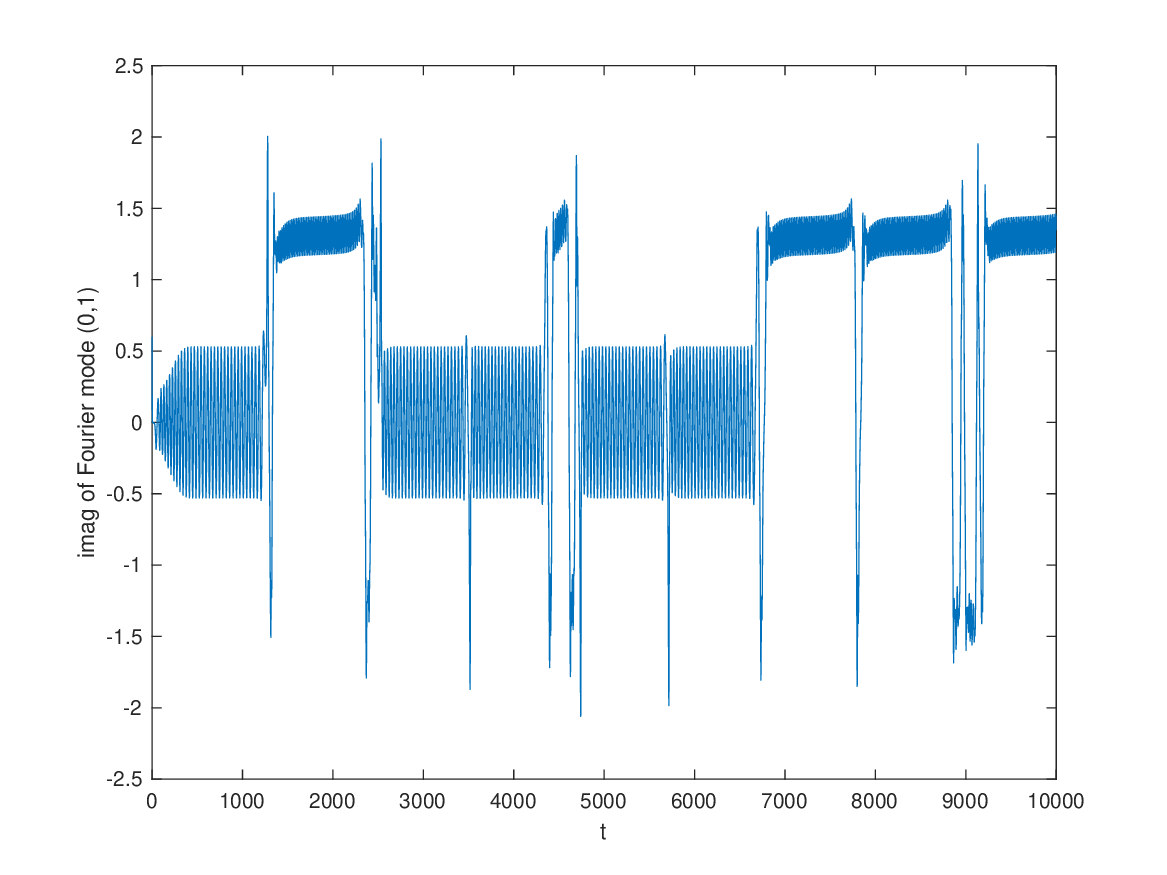}
	}
	\caption{The real and imaginary part of the Fourier coefficient of mode $e^{iy}$ as a function of time. $Re=25.7715, k=0.001$ with 256 Fourier modes.}
	\label{mode}
\end{figure}

\begin{figure}[!ht]
	\centering
	\subfigure{
		\includegraphics[width=5in]{  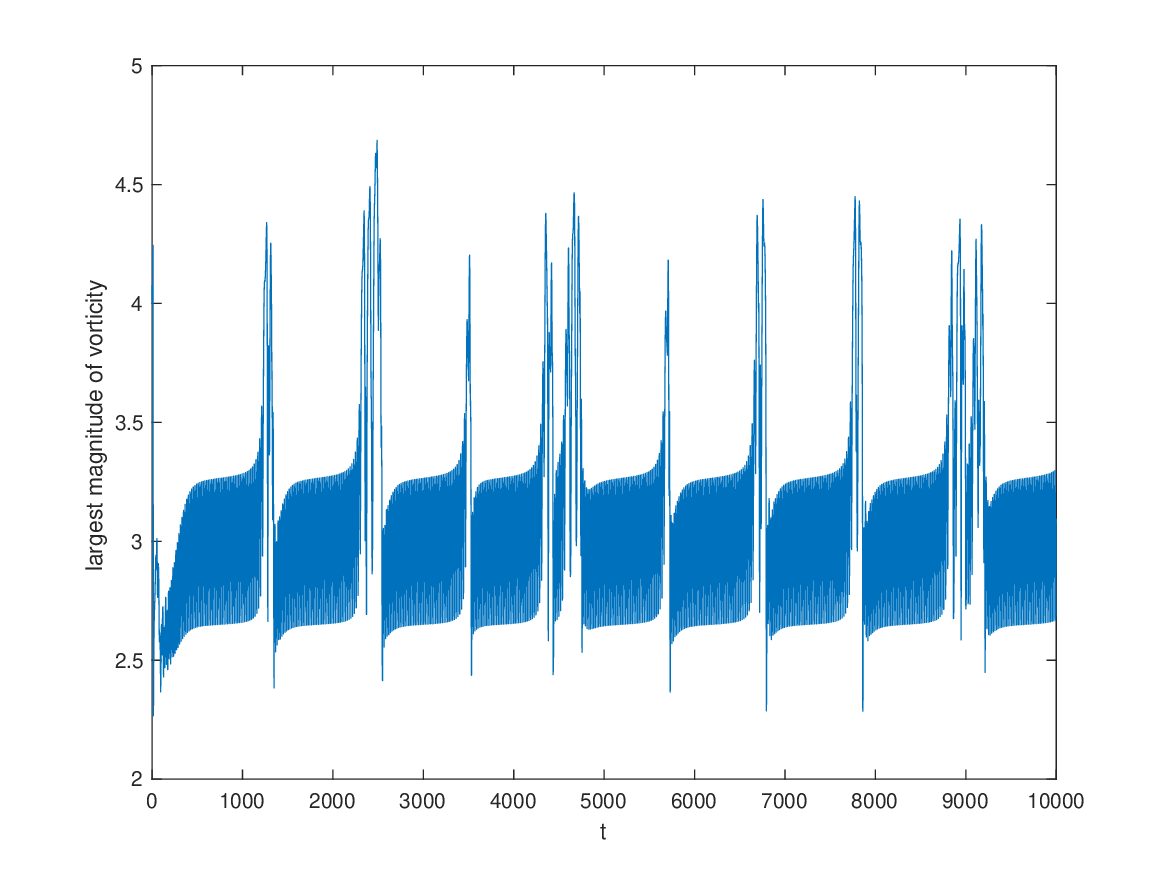}
	}
\hspace{0.5cm}
	\subfigure{
		\includegraphics[width=5in]{  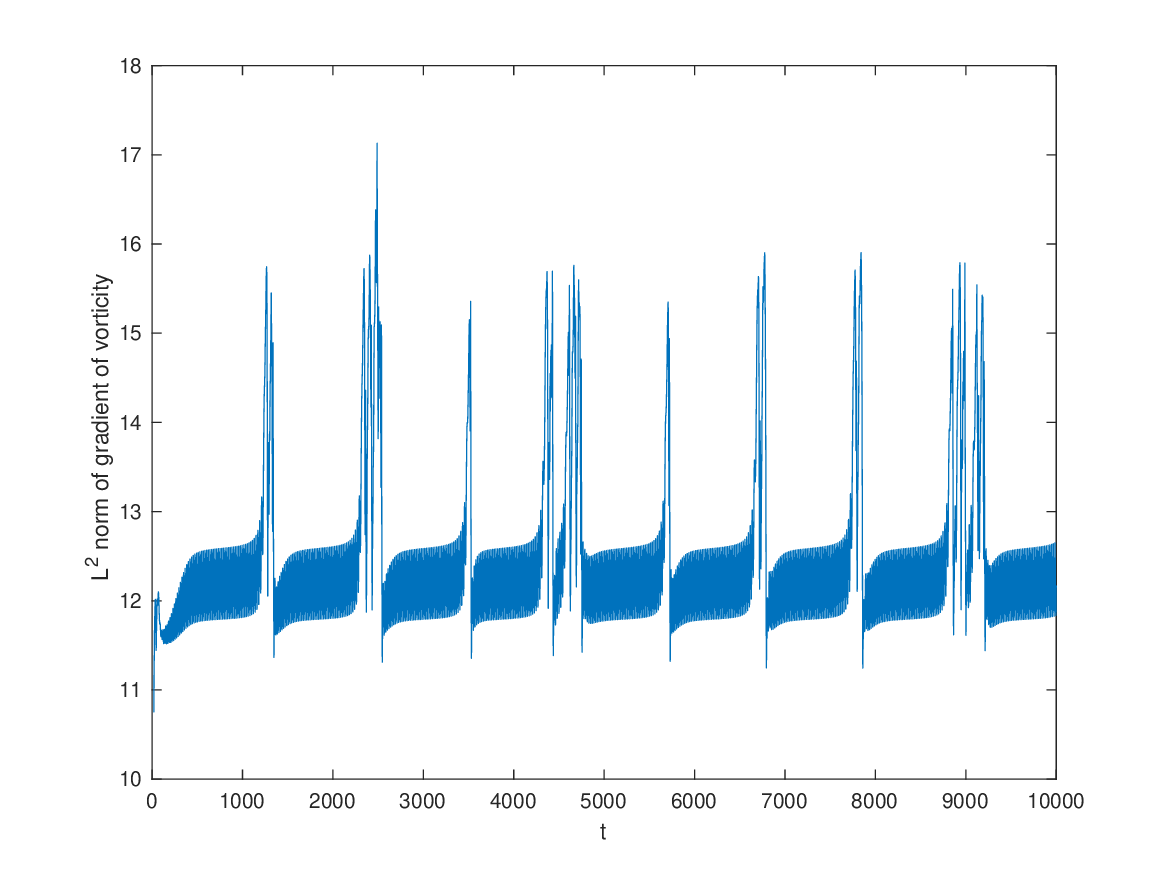}
	}
	\caption{The largest magnitude and the $L^2$ norm of  gradient of vorticity as a function of time. $Re=25.7715, k=0.001$ with 256 Fourier modes.}
	\label{L2}
\end{figure}

\begin{figure}[!ht]
	\centering	
		\includegraphics[width=5in]{  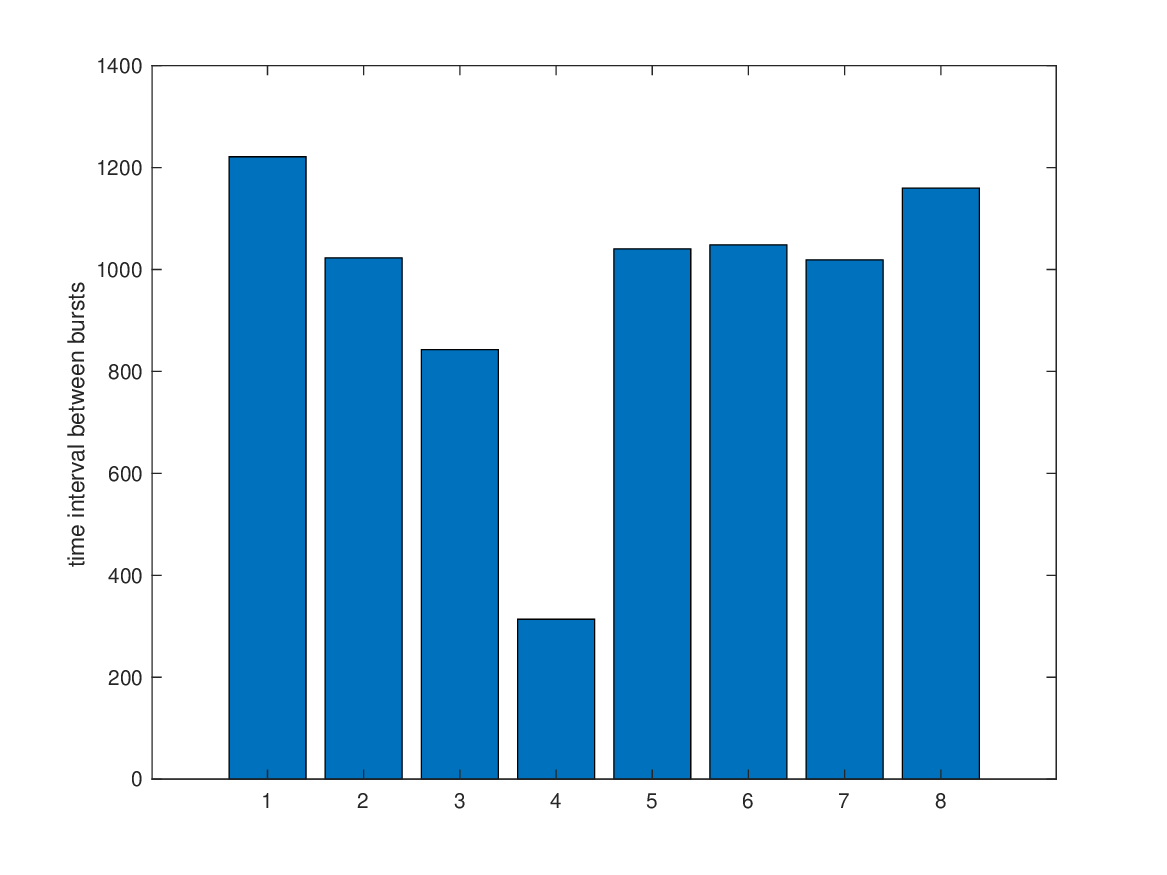}
	\caption{Time interval between bursts}
	\label{bar}
\end{figure}

\begin{figure}[!ht]
	\centering	
		\includegraphics[width=5in]{  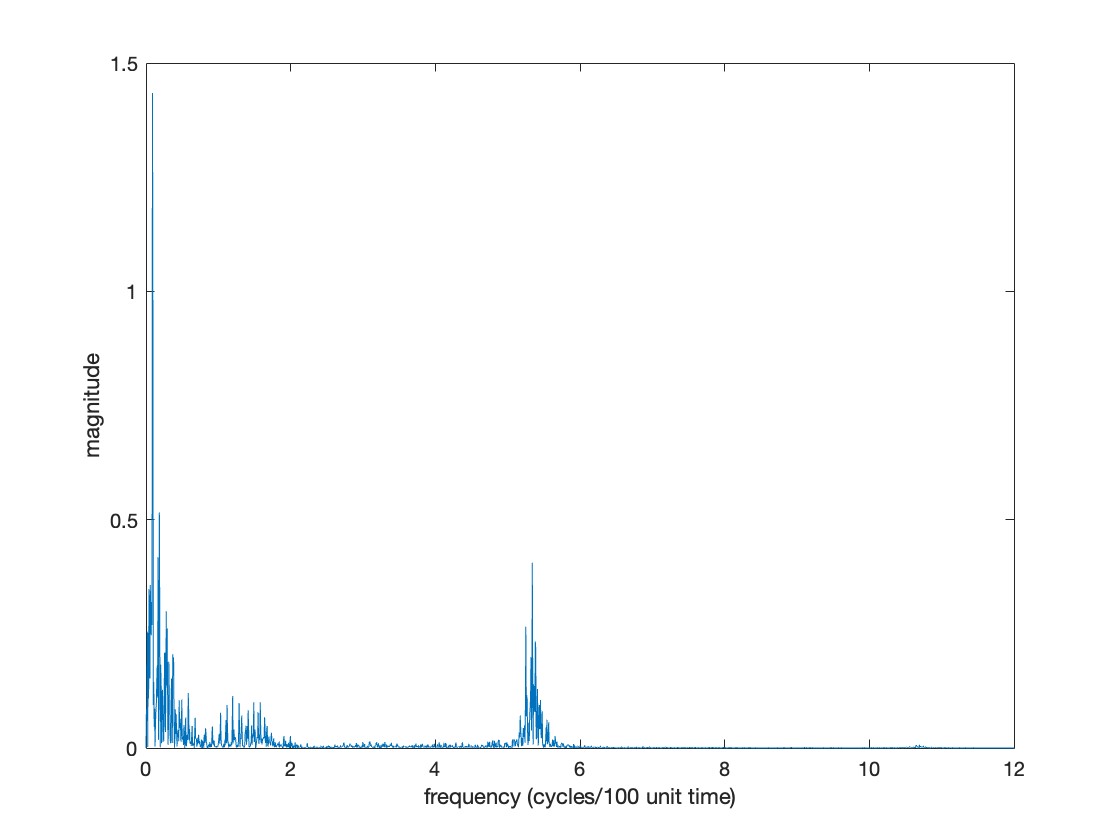}
	\caption{Power spectrum density plot}
	\label{psd}
\end{figure}

A zoomed-in plot of the maximum vorticity across the bursting event between $t=1000$ and $2000$ is displayed in Fig. \ref{par_mag}. Fig. \ref{bursting} demonstrates the typical  dynamics during  during a bursting event. It is observed that the burst is associated with  spatially localized concentration of vorticity. 
\begin{figure}[!ht]
	\centering	
		\includegraphics[width=5in]{  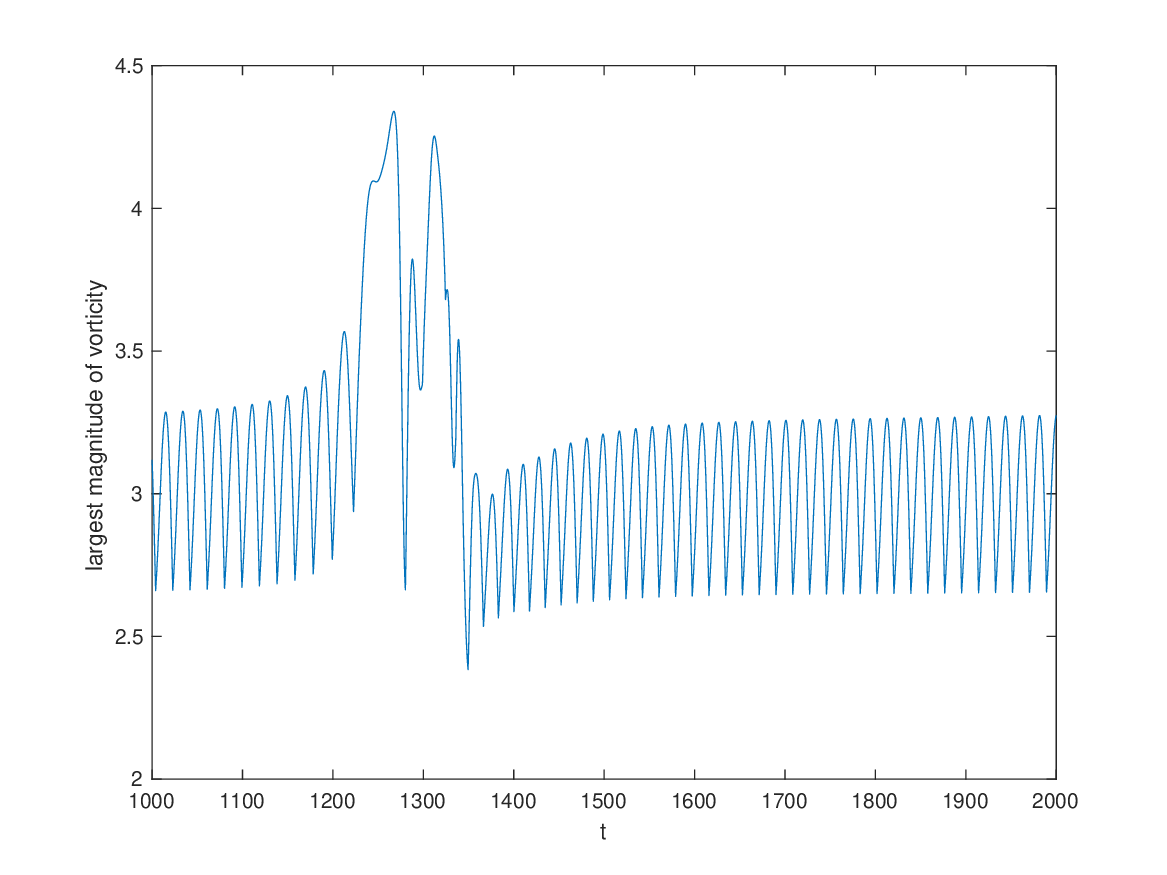}
	\caption{Evolution of the maximum magnitude of vorticity during a bursting event.}
	\label{par_mag}
\end{figure}

\begin{figure}[!ht]
			\centering
		\subfigure[ snapshots at $t=1220,1236, 1250$.]{
			\centering
			\begin{minipage}[t]{1.0\linewidth}
				\includegraphics[width=2in]{  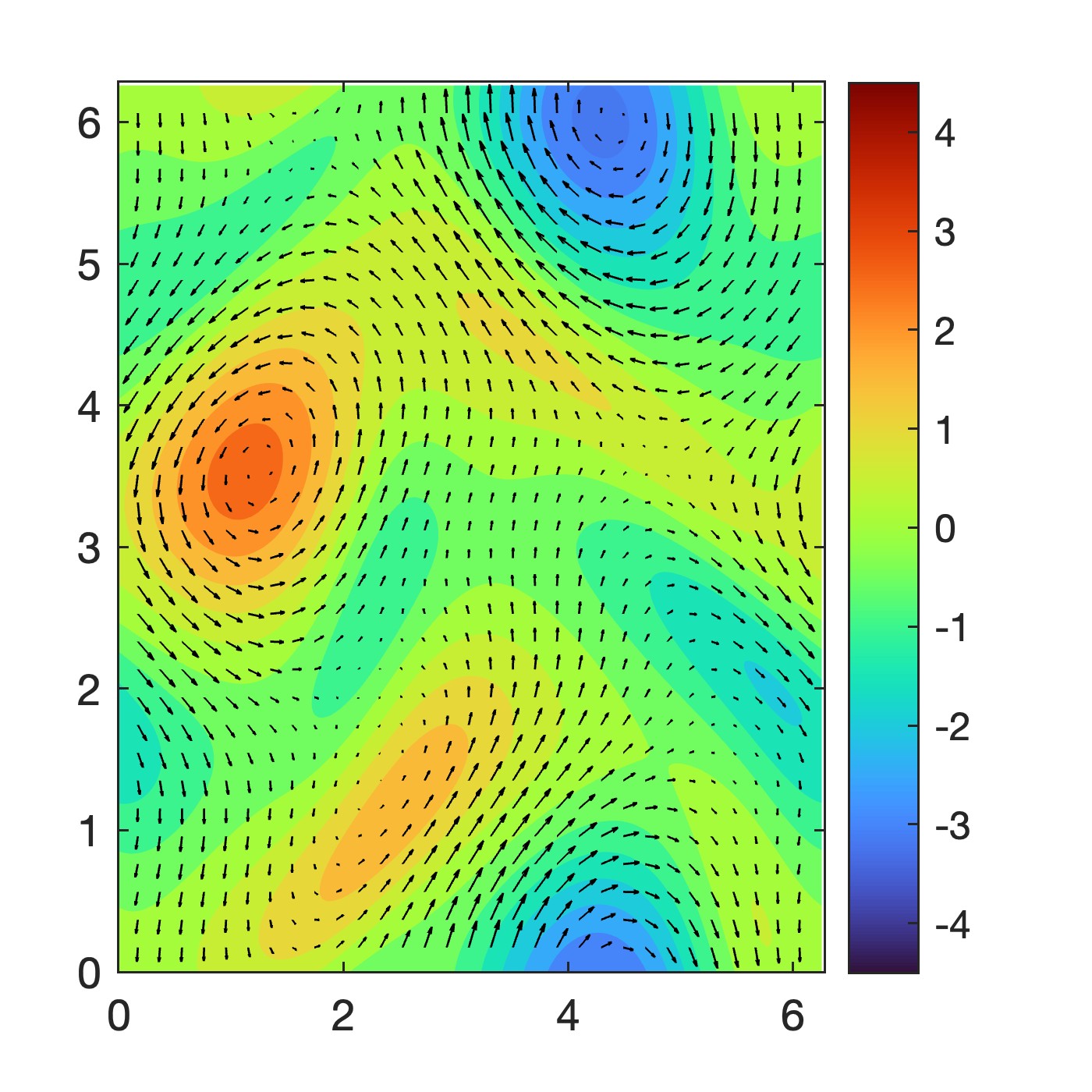}\hskip -2mm
				\includegraphics[width=2in]{  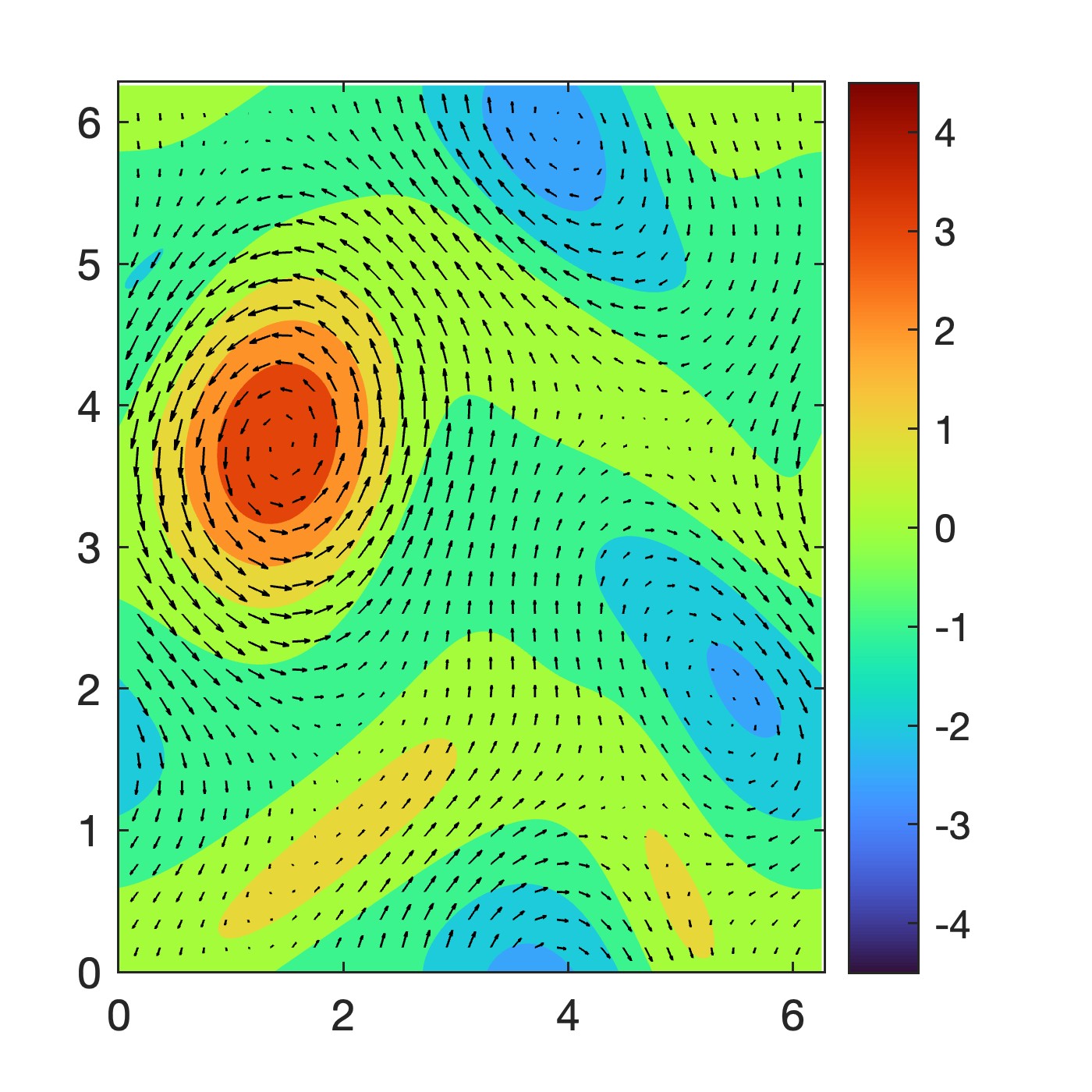}\hskip -2mm
				\includegraphics[width=2in]{  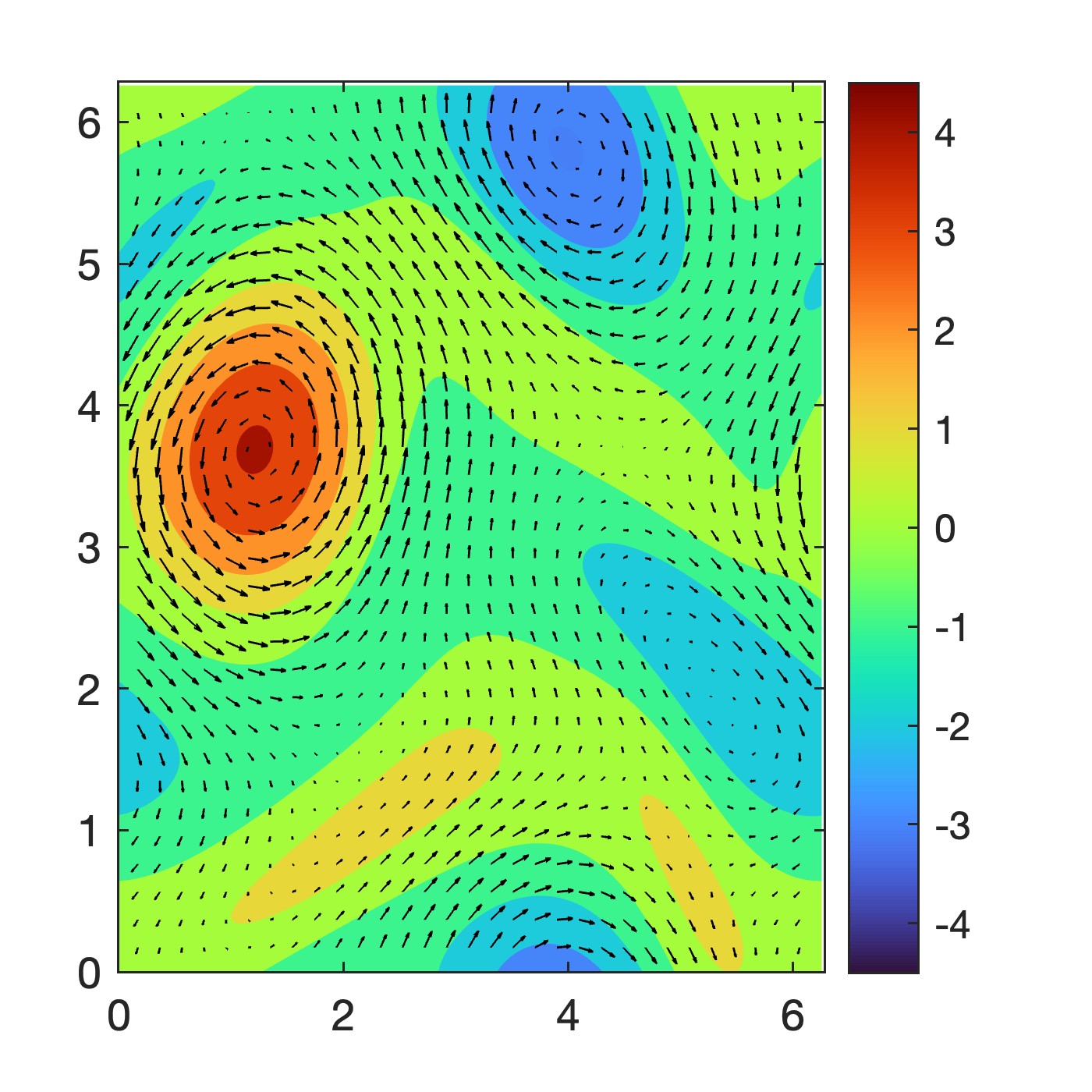}\hskip -2mm
			\end{minipage}\label{figPO:subfig:a}
		}
		\subfigure[ snapshots at $t=1260, 1275, 1400$.]{ 
			\centering
			\begin{minipage}[t]{1.0\linewidth}
              \includegraphics[width=2in]{  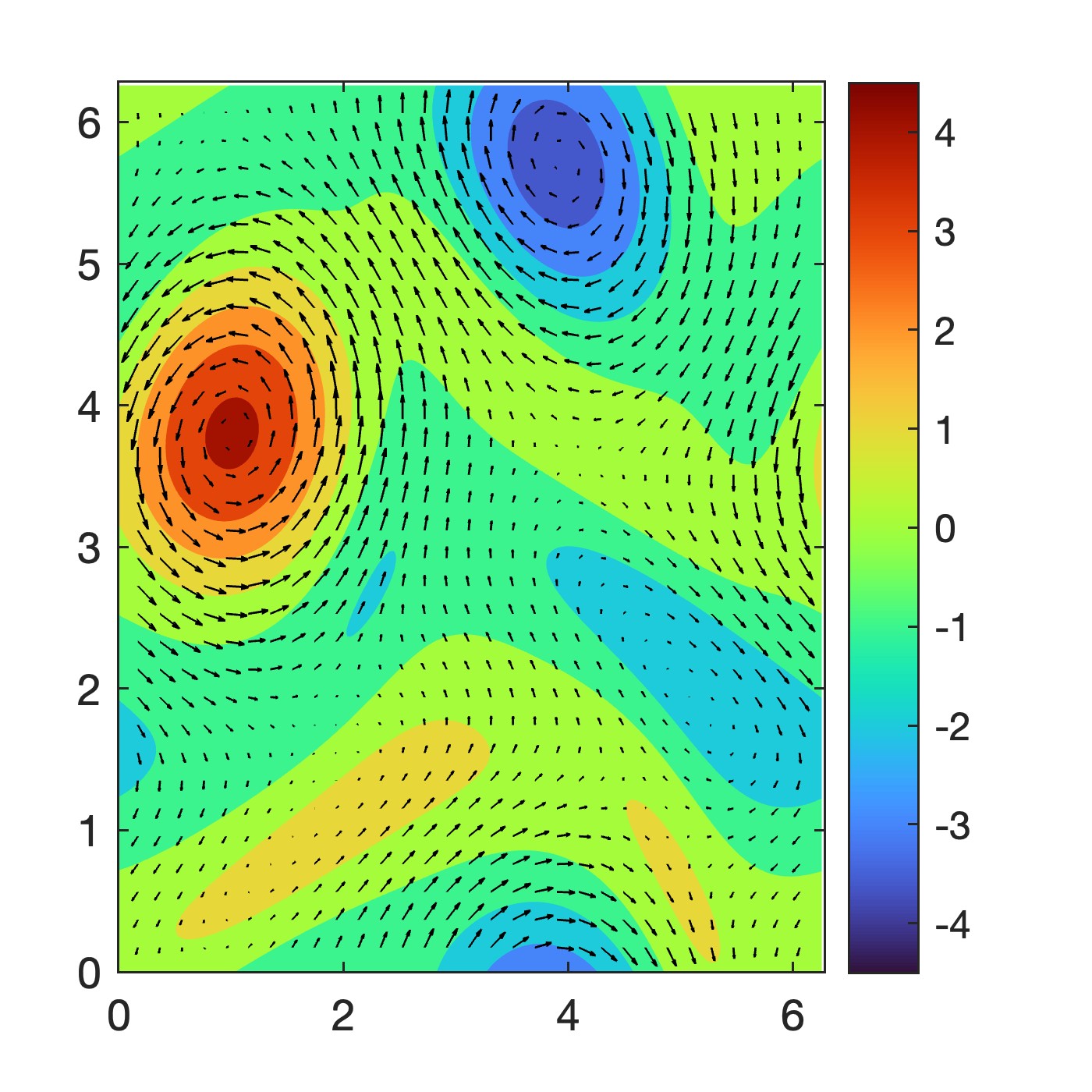} \hskip -2mm
				\includegraphics[width=2in]{  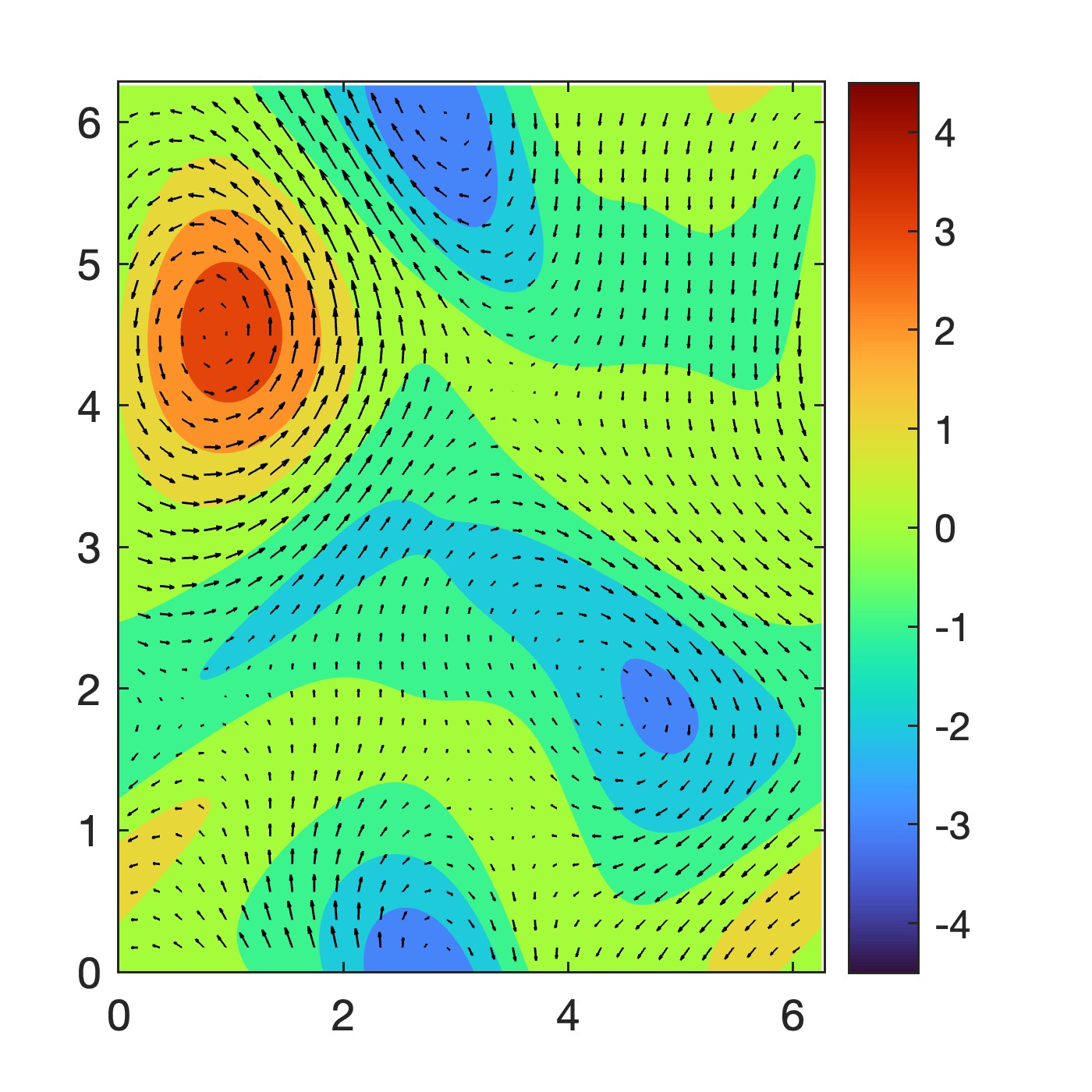} \hskip -2mm
               \includegraphics[width=2in]{  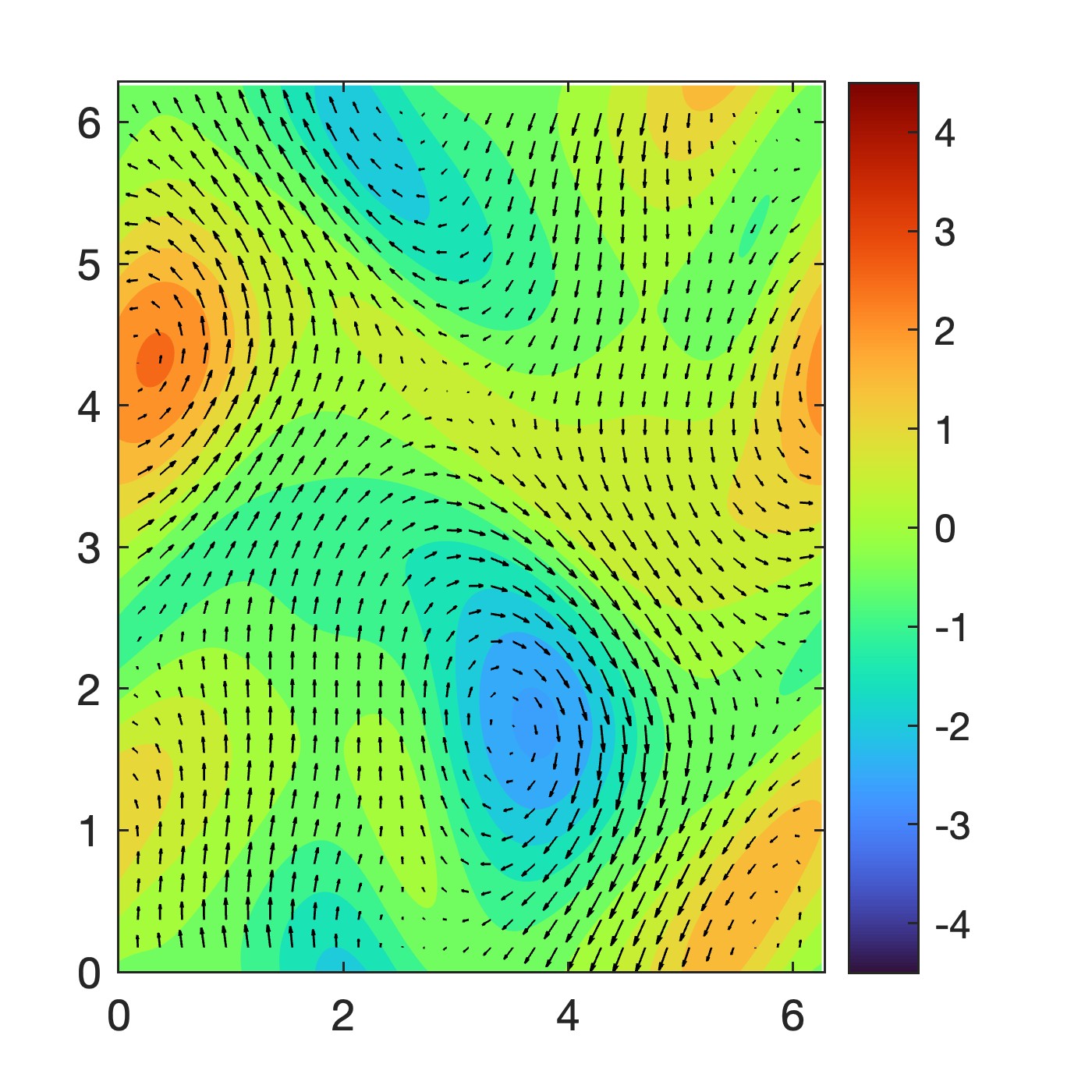}\hskip -2mm
			\end{minipage}\label{figPO:subfig:b}
		}
		\caption{The vorticity contour and velocity field during a bursting event}
		\label{bursting}
	\end{figure}

	\section{Conclusion} \label{conc}
	 A novel second-order accurate, Forced Scalar Auxiliary Variable approach (FSAV) is introduced to preserve the underlying dissipative structure of the forced Navier-Stokes system that yields a uniform-in-time estimate of the numerical solution. As an example we apply the new algorithm to the two-dimensional incompressible Navier-Stokes equations. In the case with no-penetration and free-slip boundary condition on a simply connected domain, we are also able to derive a uniform-in-time estimate of the vorticity in $H^1$ norm in addition to the $L^2$ norm guaranteed by the general framework. Numerical results demonstrate superior performance of the new algorithm in terms of accuracy, efficiency, stability and robustness. The FSAV method is applicable to a general class of forced dissipative systems with conservative nonlinear term. In addition, the numerical scheme is autonomous if the underlying model is, laying the foundation for studying long-time dynamics of the numerical solution via dynamical system approach.

	\section*{Acknowledgement}
	The work of D. Han is supported by the National Science Foundation grant DMS-2310340. The work of X. Wang is supported by the National Natural Science Foundation of China grant 12271237 as well as the Gary Havener Endowment. The authors share joint first authorship.

	
	\bibliographystyle{siam}
	\bibliography{multiphase-2024.bib}

\def\cprime{$'$}
\begin{thebibliography}{10}

\bibitem{ALL2019}
{\sc G.~Akrivis, B.~Li, and D.~Li}, {\em Energy-decaying extrapolated
  {RK}-{SAV} methods for the {A}llen-{C}ahn and {C}ahn-{H}illiard equations},
  SIAM J. Sci. Comput., 41 (2019), pp.~A3703--A3727.

\bibitem{Armbruster1996}
{\sc D.~Armbruster, B.~Nicolaenko, N.~Smaoui, and P.~Chossat}, {\em Symmetries
  and dynamics for {$2$}-{D} {N}avier-{S}tokes flow}, Phys. D, 95 (1996),
  pp.~81--93.

\bibitem{CHJ2023}
{\sc J.~Carter, D.~Han, and N.~Jiang}, {\em Second order, unconditionally
  stable, linear ensemble algorithms for the magnetohydrodynamics equations},
  J. Sci. Comput., 94 (2023), pp.~Paper No. 41, 29.

\bibitem{FMRT2001}
{\sc C.~Foias, O.~Manley, R.~Rosa, and R.~Temam}, {\em Navier-{S}tokes
  equations and turbulence}, vol.~83 of Encyclopedia of Mathematics and its
  Applications, Cambridge University Press, Cambridge, 2001.

\bibitem{Frisch1995}
{\sc U.~Frisch}, {\em Turbulence}, Cambridge University Press, Cambridge, 1995.
\newblock The legacy of A. N. Kolmogorov.

\bibitem{GiRa1986}
{\sc V.~Girault and P.-A. Raviart}, {\em Finite element methods for
  {N}avier-{S}tokes equations}, vol.~5 of Springer Series in Computational
  Mathematics, Springer-Verlag, Berlin, 1986.
\newblock Theory and algorithms.

\bibitem{GZW2020}
{\sc Y.~Gong, J.~Zhao, and Q.~Wang}, {\em Arbitrarily high-order linear energy
  stable schemes for gradient flow models}, J. Comput. Phys., 419 (2020),
  pp.~109610, 20.

\bibitem{GTWW2012}
{\sc S.~Gottlieb, F.~Tone, C.~Wang, X.~Wang, and D.~Wirosoetisno}, {\em Long
  time stability of a classical efficient scheme for two-dimensional
  {N}avier-{S}tokes equations}, SIAM J. Numer. Anal., 50 (2012), pp.~126--150.

\bibitem{GuTi2013}
{\sc F.~Guill{\'e}n-Gonz{\'a}lez and G.~Tierra}, {\em On linear schemes for a
  {C}ahn-{H}illiard diffuse interface model}, J. Comput. Phys., 234 (2013),
  pp.~140--171.

\bibitem{HOR2017}
{\sc T.~Heister, M.~A. Olshanskii, and L.~G. Rebholz}, {\em Unconditional
  long-time stability of a velocity-vorticity method for the 2{D}
  {N}avier-{S}tokes equations}, Numer. Math., 135 (2017), pp.~143--167.

\bibitem{HeRa1990}
{\sc J.~G. Heywood and R.~Rannacher}, {\em Finite-element approximation of the
  nonstationary {N}avier-{S}tokes problem. {IV}. {E}rror analysis for
  second-order time discretization}, SIAM J. Numer. Anal., 27 (1990),
  pp.~353--384.

\bibitem{HiSu2000}
{\sc A.~T. Hill and E.~S{\"u}li}, {\em Approximation of the global attractor
  for the incompressible {N}avier-{S}tokes equations}, IMA J. Numer. Anal., 20
  (2000), pp.~633--667.

\bibitem{HuSh2021}
{\sc F.~Huang and J.~Shen}, {\em Stability and error analysis of a class of
  high-order {IMEX} schemes for {N}avier-{S}tokes equations with periodic
  boundary conditions}, SIAM J. Numer. Anal., 59 (2021), pp.~2926--2954.

\bibitem{HuSh2022}
\leavevmode\vrule height 2pt depth -1.6pt width 23pt, {\em A new class of
  implicit-explicit {BDF{$k$}} {SAV} schemes for general dissipative systems
  and their error analysis}, Comput. Methods Appl. Mech. Engrg., 392 (2022),
  pp.~Paper No. 114718, 25.

\bibitem{JiYa2021}
{\sc N.~Jiang and H.~Yang}, {\em Stabilized scalar auxiliary variable ensemble
  algorithms for parameterized flow problems}, SIAM J. Sci. Comput., 43 (2021),
  pp.~A2869--A2896.

\bibitem{Ladyzhenskaya1969}
{\sc O.~A. Ladyzhenskaya}, {\em The mathematical theory of viscous
  incompressible flow}, Second English edition, revised and enlarged.
  Translated from the Russian by Richard A. Silverman and John Chu. Mathematics
  and its Applications, Vol. 2, Gordon and Breach Science Publishers, New York,
  1969.

\bibitem{LiSh2020}
{\sc X.~Li and J.~Shen}, {\em Error analysis of the {SAV}-{MAC} scheme for the
  {N}avier-{S}tokes equations}, SIAM J. Numer. Anal., 58 (2020),
  pp.~2465--2491.

\bibitem{LSL2021}
{\sc X.~Li, J.~Shen, and Z.~Liu}, {\em New {SAV}-pressure correction methods
  for the {N}avier-{S}tokes equations: stability and error analysis}, Math.
  Comp., 91 (2021), pp.~141--167.

\bibitem{LYD2019}
{\sc L.~Lin, Z.~Yang, and S.~Dong}, {\em Numerical approximation of
  incompressible {N}avier-{S}tokes equations based on an auxiliary energy
  variable}, J. Comput. Phys., 388 (2019), pp.~1--22.

\bibitem{MaWa2006a}
{\sc A.~J. Majda and X.~Wang}, {\em Non-linear dynamics and statistical
  theories for basic geophysical flows}, Cambridge University Press, Cambridge,
  2006.

\bibitem{MoYa2007}
{\sc A.~S. Monin and A.~M. Yaglom}, {\em Statistical fluid mechanics: mechanics
  of turbulence. {V}ol. {I}}, Dover Publications, Inc., Mineola, NY,
  english~ed., 2007.
\newblock Translated from the 1965 Russian original, Edited and with a preface
  by John L. Lumley, Reprinted from the 1971 edition.

\bibitem{ReTo2023}
{\sc L.~Rebholz and F.~Tone}, {\em Long-time {$H^1$}-stability of {BDF}2 time
  stepping for 2{D} {N}avier-{S}tokes equations}, Appl. Math. Lett., 141
  (2023), pp.~Paper No. 108624, 8.

\bibitem{SXY2018}
{\sc J.~Shen, J.~Xu, and J.~Yang}, {\em The scalar auxiliary variable ({SAV})
  approach for gradient flows}, J. Comput. Phys., 353 (2018), pp.~407--416.

\bibitem{SXY2019}
\leavevmode\vrule height 2pt depth -1.6pt width 23pt, {\em A new class of
  efficient and robust energy stable schemes for gradient flows}, SIAM Rev., 61
  (2019), pp.~474--506.

\bibitem{SiAr1994}
{\sc J.~C. Simo and F.~Armero}, {\em Unconditional stability and long-term
  behavior of transient algorithms for the incompressible {N}avier-{S}tokes and
  {E}uler equations}, Comput. Methods Appl. Mech. Engrg., 111 (1994),
  pp.~111--154.

\bibitem{Temam1997}
{\sc R.~Temam}, {\em Infinite-dimensional dynamical systems in mechanics and
  physics}, vol.~68 of Applied Mathematical Sciences, Springer-Verlag, New
  York, second~ed., 1997.

\bibitem{Tone2007}
{\sc F.~Tone}, {\em On the long-time stability of the {C}rank-{N}icolson scheme
  for the 2{D} {N}avier-{S}tokes equations}, Numer. Methods Partial
  Differential Equations, 23 (2007), pp.~1235--1248.

\bibitem{ToWi2006}
{\sc F.~Tone and D.~Wirosoetisno}, {\em On the long-time stability of the
  implicit {E}uler scheme for the two-dimensional {N}avier-{S}tokes equations},
  SIAM J. Numer. Anal., 44 (2006), pp.~29--40.

\bibitem{Wang2010a}
{\sc X.~Wang}, {\em Approximation of stationary statistical properties of
  dissipative dynamical systems: time discretization}, Math. Comp., 79 (2010),
  pp.~259--280.

\bibitem{Wang2012}
\leavevmode\vrule height 2pt depth -1.6pt width 23pt, {\em An efficient second
  order in time scheme for approximating long time statistical properties of
  the two dimensional {N}avier--{S}tokes equations}, Numer. Math., 121 (2012),
  pp.~753--779.

\bibitem{WaYu2024}
{\sc X.~Wang and Y.~Yu}, {\em A note on the stability of two families of
  two-step schemes}, arXiv 2406.18398,  (2024).

\bibitem{WHS2022}
{\sc K.~Wu, F.~Huang, and J.~Shen}, {\em A new class of higher-order decoupled
  schemes for the incompressible {N}avier-{S}tokes equations and applications
  to rotating dynamics}, J. Comput. Phys., 458 (2022), pp.~Paper No. 111097,
  16.

\bibitem{Yang2021}
{\sc X.~Yang}, {\em A novel fully-decoupled, second-order and energy stable
  numerical scheme of the conserved {A}llen-{C}ahn type flow-coupled binary
  surfactant model}, Comput. Methods Appl. Mech. Engrg., 373 (2021), p.~113502.

\bibitem{YaHa2017}
{\sc X.~Yang and D.~Han}, {\em Linearly first- and second-order,
  unconditionally energy stable schemes for the phase field crystal equation},
  J. Comput. Phys., 330 (2017), pp.~1116--1134.

\bibitem{YaJu2017}
{\sc X.~Yang and L.~Ju}, {\em Linear and unconditionally energy stable schemes
  for the binary fluid-surfactant phase field model}, Comput. Methods Appl.
  Mech. Engrg., 318 (2017), pp.~1005--1029.

\bibitem{YZW2017}
{\sc X.~Yang, J.~Zhao, and Q.~Wang}, {\em Numerical approximations for the
  molecular beam epitaxial growth model based on the invariant energy
  quadratization method}, J. Comput. Phys., 333 (2017), pp.~104--127.

\end{thebibliography}
	
\end{document}